\documentclass[10pt]{amsart}

\usepackage{fullpage}
\usepackage[utf8]{inputenc}
\usepackage[T2A,T1]{fontenc}
\usepackage{textcomp}
\usepackage{yfonts}

\usepackage{amsthm}
\usepackage{mathtools}
\usepackage{bm}
\usepackage{enumitem}
\usepackage{microtype}
\usepackage{amsmath}
\usepackage{amsfonts}
\usepackage{amssymb}
\usepackage{amscd}
\usepackage{stmaryrd}
\usepackage{mathrsfs}
\usepackage{mathtext}
\usepackage{accents}
\usepackage{nicefrac}

\usepackage{comment}
\usepackage{enumerate}
\usepackage{fvextra}
\usepackage{xparse}

\usepackage[dvipsnames]{xcolor}
\usepackage{tikz}
\usetikzlibrary{arrows.meta,calc}
\usepackage[all,tips]{xy}
\usepackage{soul}
\usepackage[normalem]{ulem}
\usepackage{tcolorbox}
\tcbuselibrary{breakable,skins}

\makeatletter
\@ifclassloaded{svjour3}{}{\RequirePackage{amsthm}}
\makeatother

\usepackage{hyperref}

\definecolor{OldColor}{HTML}{FF9900} %
\definecolor{NewColor}{HTML}{0000CC} %
\definecolor{MarkColor}{HTML}{FFF59D} %
\sethlcolor{MarkColor}

\newif\ifshowmarkup
\showmarkupfalse       %

\ifshowmarkup

\newenvironment{newstuff}
  {\begingroup\color{NewColor}\bfseries}
  {\endgroup}

\NewDocumentCommand{\todobox}{ O{yellow} m }{%
  \begin{tcolorbox}[
    enhanced,
    breakable,
    colback=#1!10,        %
    colframe=#1!70!black, %
    boxrule=0.8pt,
    arc=2pt,
    left=6pt,right=6pt,top=4pt,bottom=4pt
  ]
  \emph{TODO:}~#2
  \end{tcolorbox}
}

\NewDocumentCommand{\donebox}{ O{gray} m }{%
  \begin{tcolorbox}[
    enhanced,
    breakable,
    colback=#1!10,        %
    colframe=#1!70!black, %
    boxrule=0.8pt,
    arc=2pt,
    left=6pt,right=6pt,top=4pt,bottom=4pt
  ]
  \emph{DONE:}~#2
  \end{tcolorbox}
}

\NewDocumentCommand{\todoboxc}{ O{yellow!15} O{yellow!60!black} m }{%
  \begin{tcolorbox}[
    enhanced, breakable,
    colback=#1,
    colframe=#2,
    boxrule=0.8pt, arc=2pt
  ]
  \emph{TODO:}~#3
  \end{tcolorbox}
}

\else

  \excludecomment{oldstuff}

  \NewDocumentCommand{\todobox}{ O{yellow} m }{}
  \NewDocumentCommand{\donebox}{ O{gray} m }{}
  \NewDocumentCommand{\todoboxc}{ O{yellow!15} O{yellow!60!black} m }{}

\fi

\makeatletter
\@ifclassloaded{svjour3}{}{%
\newtheorem{theorem}{Theorem}[section]
\newtheorem{lemma}[theorem]{Lemma}

\newtheorem{proposition}[theorem]{Proposition}
\theoremstyle{definition}

\newtheorem{example}[theorem]{Example}
\theoremstyle{remark}
\newtheorem{remark}[theorem]{Remark}
\numberwithin{equation}{section}
}
\makeatother

\newcommand{\bbB}{{\mathbb B}}

\newcommand{\bbE}{{\mathbb E}}

\newcommand{\bbH}{{\mathbb H}}

\newcommand{\bbR}{{\mathbb R}}
\newcommand{\bbS}{{\mathbb S}}

\newcommand{\calA}{{\mathcal A}}

\makeatletter
\DeclareFontFamily{OMX}{MnSymbolE}{}
\DeclareSymbolFont{MnLargeSymbols}{OMX}{MnSymbolE}{m}{n}
\SetSymbolFont{MnLargeSymbols}{bold}{OMX}{MnSymbolE}{b}{n}
\DeclareFontShape{OMX}{MnSymbolE}{m}{n}{
    <-6>  MnSymbolE5
   <6-7>  MnSymbolE6
   <7-8>  MnSymbolE7
   <8-9>  MnSymbolE8
   <9-10> MnSymbolE9
  <10-12> MnSymbolE10
  <12->   MnSymbolE12
}{}
\DeclareFontShape{OMX}{MnSymbolE}{b}{n}{
    <-6>  MnSymbolE-Bold5
   <6-7>  MnSymbolE-Bold6
   <7-8>  MnSymbolE-Bold7
   <8-9>  MnSymbolE-Bold8
   <9-10> MnSymbolE-Bold9
  <10-12> MnSymbolE-Bold10
  <12->   MnSymbolE-Bold12
}{}

\let\llangle\@undefined
\let\rrangle\@undefined
\DeclareMathDelimiter{\llangle}{\mathopen}{MnLargeSymbols}{'164}{MnLargeSymbols}{'164}
\DeclareMathDelimiter{\rrangle}{\mathclose}{MnLargeSymbols}{'171}{MnLargeSymbols}{'171}
\makeatother

\makeatletter
\newcommand\suchthat{%
 \@ifstar
  {\mathrel{}\middle|\mathrel{}}
  {\mid}%
}
\makeatother

\newcommand{\Ball}[2]{\bbB(#1,#2)}
\newcommand{\Sphere}[2]{\bbS(#1,#2)}

\newcommand{\unitball}[1]{\hat\bbB^{#1}}
\newcommand{\unitsphere}[1]{\hat\bbS^{#1}}

\newcommand{\unitballvolume}[1]{\lvert\unitball{#1}\rvert}
\newcommand{\unitspherearea}[1]{\lvert\unitsphere{#1}\rvert} 

\newcommand{\halfspace}{\bbH}

\newcommand{\Gnomonic}{\Phi}
\newcommand{\Normalization}{\widetilde\Phi}

\newcommand{\unitu}{\hat{u}}
\newcommand{\unitv}{\hat{v}}

\newcommand{\diameter}[1]{\delta_{#1}}
\newcommand{\height}[1]{\nu_{#1}}
\newcommand{\minheight}[1]{\nu_{#1}}
\newcommand{\Area}[1]{A_{#1}}
\newcommand{\volumesymbol}{V}
\newcommand{\volume}[1]{\volumesymbol_{#1}}
\newcommand{\outradius}[1]{\varrho_{#1}}
\newcommand{\inradius}[1]{r_{#1}}

\newcommand{\solidangle}{\Omega}

\newcommand{\fullness}[1]{\mu_{#1}}
\newcommand{\radiusratio}[1]{\beta_{#1}}
\newcommand{\thickness}[1]{\tau_{#1}}
\newcommand{\aspectratio}[1]{\alpha_{#1}}
\newcommand{\meanratio}[1]{\eta_{#1}}
\newcommand{\sineofminsolidangle}[1]{\sigma_{#1}}
\newcommand{\minsolidangle}[1]{\solidangle_{#1}}

\newcommand{\cone}{\operatorname{cone}}
\newcommand{\convex}{\operatorname{convex}}
\newcommand{\linhull}{\operatorname{span}}

\newcommand{\Cjung}[1]{C_{J,#1}}
\newcommand{\Cdiameter}[1]{C_{\delta,#1}}
\newcommand{\Cinradius}[1]{C_{r,#1}}
\newcommand{\Coutradius}[1]{C_{\varrho,#1}}

\newcommand{\Jacobian}{{\mathrm D}}

\newcommand{\Hausdorff}{\mathcal{H}}

\newcommand{\tsum}{{\textstyle{\sum}}}

\newcommand{\Id}{\operatorname{Id}}

\newcommand{\allones}{\bm{1}}
\newcommand{\allzeroes}{\bm{0}}
\DeclareSymbolFont{bbold}{U}{bbold}{m}{n}
\DeclareSymbolFontAlphabet{\mathbbold}{bbold}

\begin{document}

\title{On solid angles and simplex shape measures}

\author{Martin W. Licht}
\address{TU Dresden, Zellescher Weg 12-14, 01069 Dresden, Germany}
\email{martin\_werner.licht@tu-dresden.de}

\thanks{}


\date{\today}

\begin{abstract}
    The minimum solid angle of a simplex is shown to be a shape measure in any dimension. 
    While this has been known in dimensions two and three, the higher-dimensional generalization is new. 
    We prove inequalities between the minimum solid angle of a non-degenerate simplex and its fullness.
    Up to dimension-dependent constants, 
    the fullness is bounded from above by the minimum solid angle, and it is bounded from below by the square of the minimum solid angle.
    The exponents are optimal for three and higher dimensions.
    This work therefore generalizes the minimum angle condition. 
    We also review several simplex shape measures and their relationships.
\end{abstract}

\maketitle

\section{Introduction}\label{section:intro}

How to measure the quality of triangulations has been studied extensively throughout the mathematical and quantitative sciences.
Intuitively, some triangles and tetrahedra are more \emph{regular} or \emph{well-shaped} than others.
\emph{Shape measures} formalize and quantify the quality of such shapes. 
Examples of such shape measures include the \emph{aspect ratio} and the \emph{fullness} of a simplex. 
Beyond triangles and tetrahedra, shape quality indicators evaluate simplices in higher dimensions.
In this note, we show that the \emph{minimum solid angle} is a simplicial shape measure, polynomially comparable to the aspect ratio and the fullness. While this has been known for triangles and tetrahedra, the generalization to simplices of arbitrary dimension seems to have been absent from the literature up to now.

Notions of shape measures appear throughout the mathematical literature.  
Shape measures associate a real number to a simplex, 
continuous in the coordinates of the vertices and invariant under reflection, rotation, translation, and positive scaling. They are also independent of the order of the vertices. 
For example, an $n$-dimensional simplex $S$ with diameter $\diameter{S}$, volume $\volume{S}$, and inradius $\inradius{S}$ has the \emph{aspect ratio} $\aspectratio{S} := \inradius{S} / \diameter{S}$ and the \emph{fullness} $\fullness{S} := \volume{S} / \diameter{S}^{n}$.
Clearly, the aspect ratio and the fullness are invariant under orthogonal transformations, translations, and positive rescaling of simplices.
If the simplex shows any type of degeneracy, then the shape measures deteriorate,
but if the simplex has a regular shape, then the shape measures remain at values close to the optimum.

In two dimensions, the minimum angle $\gamma_{\min}$ of a triangle $S$ is a shape measure. 
It is easy to verify that this shape measure is equivalent to the fullness, in the sense that  
\begin{align*} 
    \tfrac{1}{4} \gamma_{\min} 
    \leq 
    \fullness{S} 
    \leq 
    \tfrac{1}{2} \gamma_{\min}
    .
\end{align*}
It is less elementary how this shape measure generalizes to tetrahedra or higher-dimensional simplices. 
For example, 
neither a tetrahedron's minimum angle within the faces nor its minimum dihedral angle between faces along an edge would be adequate as a shape measure to detect all tetrahedron degeneracies; see Figure~\ref{figure:simplex-degeneracies}. 
The natural generalization of the minimum angle of a triangle is the \emph{minimum solid angle} at the vertices of a simplex.

The classical angle between two linearly independent unit vectors in the plane is the length of the shorter arc on the unit circle connecting them.
Solid angles generalize this idea to any dimension. 
The solid angle of three linearly independent unit vectors in three-dimensional Euclidean space is the surface area of the spherical triangle spanned by those three vectors. 
In general, the solid angle of $n$ vectors $x_1, \dots, x_{n}$ in $\bbR^{n}$ is the $(n-1)$-dimensional measure 
\begin{align*} 
	\solidangle( x_1, \dots, x_{n} )
	:=
    \Hausdorff^{n-1} ( \cone\{ x_1,\dots,x_{n} \} \cap \unitsphere{n-1} )
    ,
\end{align*}
where $\cone\{ x_1, \dots, x_{n} \} := \tsum_{i=1}^{n} \bbR_{0}^+ \cdot x_{i}$ is the convex cone spanned by $x_1, \dots, x_{n}$ and $\unitsphere{n-1}$ denotes the unit sphere of dimension $n-1$. 
If $x_1, \dots, x_{n}$ are linearly independent unit vectors, then their solid angle equals the measure of the spherical simplex with vertices $x_1, \dots, x_{n}$.
Given a simplex with vertices $v_{0}, \dots, v_n \in \bbR^{n}$, the solid angle at the vertex $v_{i}$ is defined as 
\begin{align*}
    \solidangle_{i} 
    := 
    \solidangle( 
        v_{0} - v_{i}, 
        \dots, 
        v_{i-1} - v_{i}, 
        v_{i+1} - v_{i}, 
        \dots, 
        v_{n} - v_{i} 
    ).
\end{align*}

We examine the minimum solid angle as a simplex shape measure in any dimension. 
Clearly, the solid angles remain invariant under reflection, rotation, translation, and positive scaling. 
Well-shaped simplices have large solid angles: the minimum solid angle is bounded from below by fullness of the simplex, up to a dimension-dependent constant.  
The converse has not been established yet in the literature for higher dimensions; 
it is easily seen in two dimensions, and Liu and Joe have have proven it in three dimensions~\cite{liu1994relationship}, a considerably more delicate inequality. 
Generalizing previous work, we show that simplices in any dimension with large solid angles are also well-shaped. 
We provide polynomial inequalities with explicit dimension-dependent constants and demonstrate that the exponents are optimal. 

The main result is the following theorem.

\begin{theorem}\label{intro:main-result}
    Let $S$ be an $n$-dimensional simplex, $n \geq 2$, with diameter $\diameter{S}$, volume $\volume{S}$, and minimum solid angle $\minsolidangle{S}$. 
    Write $\Gamma$ for the gamma function.
    Then 
    \begin{align}
        C_{n}^{-1} 
        \minsolidangle{S}^{2}
        \leq 
        \fullness{S}
        \leq 
        \tfrac{1}{n} \minsolidangle{S}
        ,
        \qquad 
        C_{n}
        := 
        n 
        \pi^{\frac{ n-1 }{ 2 }}
        \sum_{ \ell = 1 }^{n-2}
        \tbinom{ n-1 }{ \ell }^{2} %
        \frac{ \Gamma( \frac{ \ell+1 }{ 2 } ) \Gamma( \frac{ n - \ell }{ 2 } ) }{ \Gamma(\frac{n}{2})^{2} }
        +
        \frac{ 2 n \pi^{\frac n 2} }{ (n+1) \Gamma(\tfrac{n}{2}) }
        .
    \end{align}
    In the case $n=3$, we additionally have 
    \begin{align}
        \bigl( 12 + 3 \pi \bigr)^{-1}
        \minsolidangle{S}^{2}
        \leq 
        \fullness{S}
        \leq 
        \tfrac{1}{3} \minsolidangle{S}
        .
    \end{align}
    In dimension three or more, the exponents are optimal.
\end{theorem}

We briefly review how shape measures are used in different areas of mathematics, engineering, and natural sciences, both in application-oriented and theoretical endeavors.

Some shape measures appear throughout geometric analysis and differential geometry.
The \emph{fullness} of an $n$-dimensional simplex compares its $n$-dimensional measure against the $n$-th power of its diameter~\cite{whitney1957geometric}, 
and the \emph{thickness} of a simplex is the ratio of its smallest height to its diameter~\cite{munkres1966elementary,dyer2015riemannian}. 

Several shape measures have been proposed throughout the literature of computational geometry and numerical analysis 
to quantify the quality of simplices~\cite{liu1994relationship,dompierre2005analysis}. 
One example is the ratio of the diameter to the radius of the largest inscribed sphere of a simplex, sometimes known as the \emph{chunkiness parameter}~(see~\cite[p.~97, Definition~(4.2.16)]{brenner2008mathematical} and~\cite[Definition~11.2]{ernguermond1}). 
\emph{Algebraic shape measures} are very natural mesh quality indicators for the error analysis of numerical methods: 
simplices are associated with the affine transformations from regular reference simplices, 
and their shape quality can therefore be assessed through the properties of the corresponding Jacobian matrices. 
For example, the spectral condition number of such an affine transformation's Jacobian is a shape measure. 
While the minima for either face angles or for dihedral angles in a tetrahedron alone are not adequate shape measures, 
the minima of both face and dihedral angles is a shape measure~\cite{brandts2008equivalence}. 
Naturally, none of these shape measures depend on the orthogonal coordinate system and the units of distance.

Tetrahedral meshes are widely used in the computational solution of partial differential equations by the finite element method, which is a staple in computational engineering. 
Ensuring that the tetrahedra are well-shaped is a central problem for the initial mesh generation and for driving adaptive mesh refinement:
poorly shaped tetrahedra can adversely affect the discretization errors of the numerical methods, the conditioning of numerical problems, and the accuracy and stability of the numerical solvers. 
We refer to the literature for overviews on triangular and tetrahedral shape measures~\cite{dompierre2005analysis,brandts2008equivalence,coxeter_triangulations}.
Beyond three dimensions, space-time methods~\cite{neumuller2011refinement,steinbach2015space} posed on four-dimensional simplicial meshes give us reason to study shape measures in higher dimensions. 
This manuscript establishes that a family of simplicial meshes in any dimension is shape-regular if and only if the minimum solid angles among all the simplices have a common lower bound.

The notion of solid angle has been known at least since Euclid's \emph{Elements}~\cite{Negrepontis2026}.
The connection to spherical geometry in three or more dimensions has been established over the past centuries~\cite{Girard1629,euler1781mensura,schlaefli,SEN_1930}.
Formulas for solid angles exist in the three-dimensional case, yet no exact elementary expressions in terms of the vector coordinates are known in higher dimensions~\cite{eriksson1990measure,hajja_walker,ribando2006measuring}. 
Solid angles can be expressed in terms of integrals of multivariate Gaussian probability distributions over the positive orthant. 
Whilst not as commonplace as planar angles, solid angles have remained a staple in several disciplines of engineering and science.
The minimum angle is easily seen to be a shape measure for triangles, 
and the minimum solid angle is a shape measure for tetrahedra~\cite{liu1994relationship,brandts2009equivalence}.
We generalize these observations to arbitrary dimension.

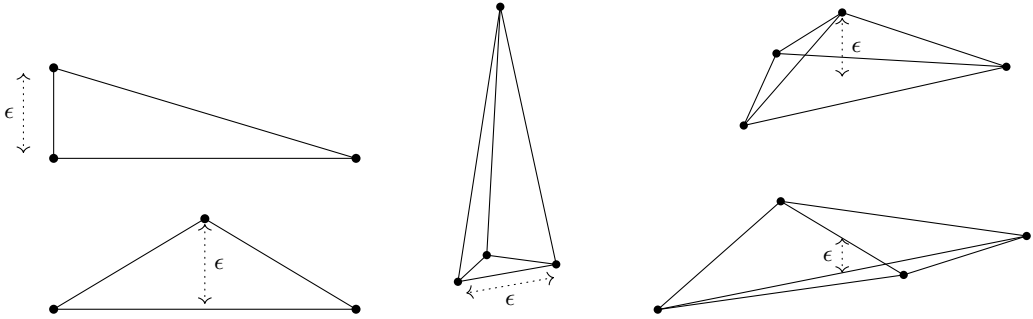
\begin{figure}[t]
    \centering

    \begin{tikzpicture}[
        scale=4.0,
        line cap=round,
        line join=round,
        vertex/.style={circle,fill,inner sep=1.2pt},
        every node/.style={font=\small}
    ]
        \begin{scope}[yshift=0.50cm]
            \coordinate (A) at (0,0);
            \coordinate (B) at (0,0.30);
            \coordinate (C) at (1,0);

            \draw (A)--(B)--(C)--cycle;
            \node[vertex] at (A) {};
            \node[vertex] at (B) {};
            \node[vertex] at (C) {};

            \draw[dotted,<->,shorten <=2pt,shorten >=2pt] (-0.10,0) -- (-0.10,0.30) node[midway,left] {$\epsilon$}; 
        \end{scope}
        \begin{scope}
            \coordinate (A) at (0,0);
            \coordinate (B) at (1,0);
            \coordinate (C) at (0.5,0.30);

            \draw (A)--(B)--(C)--cycle;
            \node[vertex] at (A) {};
            \node[vertex] at (B) {};
            \node[vertex] at (C) {};

            \draw[dotted,<->,shorten <=2pt,shorten >=2pt] (0.50,0) -- (0.50,0.30) node[midway,right] {$\epsilon$}; %
        \end{scope}
    \end{tikzpicture}
    \hspace{1cm}
    \begin{tikzpicture}[
        x={(1.00cm,0.00cm)},
        y={(0.50cm,0.30cm)},
        z={(0cm,1cm)},
        scale=1.15,
        line cap=round,
        line join=round,
        vertex/.style={circle,fill,inner sep=1.1pt},
        every node/.style={font=\small}
    ]
        \begin{scope} 
            \def\needleangle{10}
            \def\needleheight{3.0}
            \def\baseradius{0.60}

            \coordinate (A) at ({\baseradius*cos(\needleangle    )},{\baseradius*sin(\needleangle    )},0);
            \coordinate (B) at ({\baseradius*cos(\needleangle+120)},{\baseradius*sin(\needleangle+120)},0);
            \coordinate (C) at ({\baseradius*cos(\needleangle+240)},{\baseradius*sin(\needleangle+240)},0);
            \coordinate (D) at (0,0,\needleheight);

            \draw (A)--(B)--(C)--cycle;
            \draw (A)--(D)--(B);
            \draw (C)--(D);

            \draw[dotted,<->,shorten <=2pt,shorten >=2pt] ($(A)+(1pt,-4pt)$) -- ($(C)+(1pt,-4pt)$) node[midway,below] {$\epsilon$};

            \foreach \P in {A,B,C,D} \node[vertex] at (\P) {};

        \end{scope}
    \end{tikzpicture}
    \hspace{1cm}
    \begin{tikzpicture}[
        x={(1.00cm,0.00cm)},
        y={(0.50cm,0.30cm)},
        z={(0cm,1cm)},
        scale=1.15,
        line cap=round,
        line join=round,
        vertex/.style={circle,fill,inner sep=1.1pt},
        every node/.style={font=\small}
    ]

        \begin{scope}[yshift=1.00cm]
            \def\baseradius{1.70}
            \def\flatteningangle{140}
            \def\flatteningheight{0.8}
            
            \coordinate (A) at ({\baseradius*cos(\flatteningangle    )},{\baseradius*sin(\flatteningangle    )},0);
            \coordinate (B) at ({\baseradius*cos(\flatteningangle+120)},{\baseradius*sin(\flatteningangle+120)},0);
            \coordinate (C) at ({\baseradius*cos(\flatteningangle+240)},{\baseradius*sin(\flatteningangle+240)},0);
            \coordinate (D) at (0,0,\flatteningheight);

            \draw (A)--(B)--(C)--cycle;
            \draw (A)--(D)--(B);
            \draw (C)--(D);

            \foreach \P in {A,B,C,D}
                \node[vertex] at (\P) {};

            \draw[dotted,<->,shorten <=2pt,shorten >=2pt] (0,0,0)--(D) node[midway,right] {$\epsilon$};

        \end{scope}

        \begin{scope}[yshift=-1.00cm]
            \def\baseradius{2.00}
            \def\wedgeangle{135}
            \def\wedgeheight{0.2}
            
            \coordinate (A) at ({\baseradius*cos(\wedgeangle    )},{\baseradius*sin(\wedgeangle    )}, \wedgeheight);
            \coordinate (B) at ({\baseradius*cos(\wedgeangle+ 90)},{\baseradius*sin(\wedgeangle +90)},-\wedgeheight);
            \coordinate (C) at ({\baseradius*cos(\wedgeangle+180)},{\baseradius*sin(\wedgeangle+180)}, \wedgeheight);
            \coordinate (D) at ({\baseradius*cos(\wedgeangle+270)},{\baseradius*sin(\wedgeangle+270)},-\wedgeheight);

            \draw (A)--(B);
            \draw (B)--(C);
            \draw (C)--(D);
            \draw (D)--(A);

            \draw (A)--(C);
            \draw (B)--(D);

            \foreach \P in {A,B,C,D} \node[vertex] at (\P) {};

            \draw[dotted,<->] (0,0, \wedgeheight) -- (0,0,-\wedgeheight) node[midway,left] {$\epsilon$};
        \end{scope}
    \end{tikzpicture}
    \caption{
        Representative examples of simplices degenerating as $\epsilon \to 0$. 
        Left: triangles degenerate to a line. 
        Middle: a needle tetrahedron degenerating to a single line.
        Right: tetrahedra degenerating to a triangle (top, pyramid) or a quadrilateral (bottom, sliver). 
        Tetrahedra may degenerate even if dihedral or face angles remain bounded from above and below, but the degeneracy is detected by the minimum solid angle.
    }
    \label{figure:simplex-degeneracies}
\end{figure}

The remainder of this manuscript has the following structure. 
Section~\ref{section:shapemeasures} reviews common simplex shape measures and relationships between them.
Section~\ref{section:solidangles} defines solid angles and discusses the gnomonic projection. 
Section~\ref{section:main-result} shows that the minimum solid angle is a shape measure of simplices in dimension three and higher.

\subsection*{Notation}

When the dimension $n$ is understood, $e_1, \dots, e_{n} \in \bbR^{n}$ are the standard orthonormal basis vectors. 
Write $\allones_{n}$ and $\allzeroes_{n}$ for the $n$-dimensional vectors whose entries are all equal to $1$ or $0$, respectively.
$\Id_{m \times n}$ is the $m \times n$ matrix whose diagonal entries equal $1$ and whose off-diagonal entries are zero. Set $\Id_{n} := \Id_{n \times n}$. 

Given $x \in \bbR^{n}$ and $r > 0$, 
let $\Ball{x}{r}$ be the $n$-dimensional open ball around $x$ of radius $r$,
and let $\Sphere{x}{r}$ be the sphere of dimension $n-1$ around $x$ of radius $r$.
Write $\unitball{n} := \Ball{\allzeroes_{n}}{1}$ and $\unitsphere{n-1} := \Sphere{\allzeroes_{n}}{1}$ for the unit ball and the unit sphere in $\bbR^{n}$, respectively.
Write $\unitballvolume{n} := \Hausdorff^{n}(\unitball{n})$ 
and $\unitspherearea{n-1} := \Hausdorff^{n-1}(\unitsphere{n-1})$ 
for their respective measures. 
We adopt the convention that the zero-dimensional measure of a point is one, and that the measures are understood in the Hausdorff sense.

Set $\bbR^+ := (0,\infty)$ and $\bbR_{0}^+ := [0,\infty)$. If $a_1, \dots, a_m$ are vectors, then $\linhull\{ a_1, \dots, a_m \}$ is their linear hull, $\cone\{ a_1, \dots, a_{m} \} := \tsum_{i=1}^{m} \bbR_{0}^+ \cdot a_{i}$ denotes their cone (conal hull), and $\convex\{ a_1, \dots, a_m \}$ is their convex hull.

\section{Simplices and shape measures}\label{section:shapemeasures}

A \emph{simplex} of dimension $n$ is the convex hull of $n+1$ affinely independent points within Euclidean space. 
These $n+1$ points are the \emph{vertices} of the simplex. 
By definition, no vertex is a convex combination of the others.
A simplex is called \emph{regular} if all its edges have the same length.
We henceforth assume that $n \geq 1$.

Suppose that $S$ is an $n$-dimensional simplex with vertices $v_{0}, \dots, v_{n}$. 
We write $\diameter{S}$ for its \emph{diameter}, and we write $\volume{S}$ for its $n$-dimensional \emph{volume}. 
The \emph{outer radius} $\outradius{S}$ of the simplex $S$ is the minimum radius of any closed ball containing the simplex\footnote{We distinguish the outer radius of $S$ from its \emph{circumradius}, which is the radius of the unique sphere in the affine hull of $S$ that includes all its vertices.}. 
The \emph{inradius} $\inradius{S}$ of the simplex $S$ is the maximum radius of any closed ball (within the affine hull of $S$) contained in the simplex. 
The \emph{height} of the $i$-th vertex is the distance from $v_{i}$ to the affine hull of $v_{0}, \dots, v_{i-1}, v_{i+1}, \dots, v_{n}$. 
Let $\minheight{S}$ be the smallest among the heights of the $n+1$ vertices. 

We work with several shape measures. 
The \emph{aspect ratio} $\aspectratio{S}$, the \emph{thickness} $\thickness{S}$, the \emph{radius ratio} $\radiusratio{S}$, and the \emph{fullness} $\fullness{S}$ are 
\begin{align*}
    \aspectratio{S}  := \frac{ \inradius{S} }{ \diameter{S} }
    ,
    \qquad 
    \thickness{S}  := \frac{ \minheight{S} }{ \diameter{S} }
    ,
    \qquad 
    \radiusratio{S}  := \frac{ \inradius{S} }{ \outradius{S} }
    ,
    \qquad 
    \fullness{S} := \frac{ \volume{S} }{ \diameter{S}^{n} }.
\end{align*}
All these measures describe how close a simplex is to being degenerate. 
To compare these measures, we first recall several elementary geometric identities.
For tighter inequalities, some extremal properties of the regular simplex are collected.

\begin{proposition}\label{proposition:extremal-properties}
    Let $S$ be an $n$-dimensional simplex.
    \begin{itemize}
    \item 
    If $\diameter{S} = 1$, then 
    $\outradius{S} \leq \Cjung{n} := \sqrt{\frac{n}{2(n+1)}}$ 
    and 
    $\volume{S} \leq \Cdiameter{n} := \frac{1}{n!} \sqrt{\frac{n+1}{2^{n}}}$. 
    \item 
    If $\inradius{S} = 1$, then 
    $\volume{S} \geq \Cinradius{n} := \frac{\sqrt{n}^{n}\sqrt{n+1}^{n+1}}{n!}$.
    \item 
    If $\outradius{S} = 1$, then 
    $\volume{S} \leq \Coutradius{n} := \frac{\sqrt{n+1}^{n+1}}{n!\sqrt{n}^{n}}$.
    \end{itemize}
    In each case, equality holds if $S$ is regular. 
\end{proposition}
\begin{proof}
    Jung's inequality~\cite{jung1901kleinste} states that $n$-simplices with unit diameter have outer radius at most $\Cjung{n}$. 
    The other inequalities are known extremal properties of the regular simplex;
    see~\cite[Remark 1.2]{klein_wessler}, %
   ~\cite[Lemma~11]{coxeter_triangulations}, %
   ~\cite[p.~438]{brass_moser_pach}, %
   ~\cite[Theorem~2]{basit_langi}, %
   ~\cite[Theorem~2.2]{gerber1975orthocentric}. %
\end{proof}

It is clear that $\outradius{S} \leq \diameter{S} \leq 2 \outradius{S}$, but the comparison of the outer radius and the diameter can be improved via Jung's theorem:
\begin{gather}\label{math:diameter-vs-outradius}
    \Cjung{n}^{-1} \outradius{S} 
    \leq \diameter{S} \leq 2 \outradius{S}.
\end{gather}

Let $\height{0}, \dots, \height{n}$ denote the heights of the vertices.
Let $\Area{0}, \dots, \Area{n}$ be the measures of the codimension-one faces opposite the respective vertices. 
Recall 
\begin{gather}\label{math:volume-from-height-and-face}
    \volume{S} = \frac{\Area{i} \height{i}}{n}
    .
\end{gather}
The largest ball contained within $S$ touches the codimension-one faces tangentially. 
Its center is called the \emph{incenter} and is equidistant to the affine hulls of the codimension-one faces.
Applying~\eqref{math:volume-from-height-and-face} to each simplex in the simplicial subdivision centered at the incenter, we find 
\begin{align*}
    \volume{S} = \sum_{i=0}^{n} \frac{ \inradius{S} \Area{i} }{n} = \inradius{S} \sum_{i=0}^{n} \frac{ \Area{i} }{n} = \inradius{S} \sum_{i=0}^{n} \frac{ \volume{S} }{ \height{i} }
    .
\end{align*}
This shows that $\inradius{S}$ equals the harmonic mean of the $n+1$ vertex heights divided by $n+1$:
\begin{gather}\label{math:inradius-via-heights-formula}
    \inradius{S} = \left( \sum_{i=0}^{n} \frac{1}{\height{i}} \right)^{-1}
    .
\end{gather}
In particular, $\inradius{S} \leq \minheight{S}$. 
However, the reciprocal heights $\height{i}^{-1}$ are the norms of the gradients of the barycentric coordinates of the simplex, and the sum of those gradients equals zero. Therefore, $\minheight{S}^{-1}$ is at most the sum of the other reciprocal heights. Therefore,
\begin{align*}
    \minheight{S}^{-1} \leq \inradius{S}^{-1} - \minheight{S}^{-1}
    .
\end{align*}
We conclude that 
\begin{gather}\label{math:inradius-vs-minheight} %
    \frac{1}{n+1} \minheight{S} \leq \inradius{S} \leq \frac{1}{2} \minheight{S} 
    .
\end{gather}
This naturally leads to our first set of comparisons between shape measures:
\begin{proposition}
    We have 
    \begin{gather*}
        \frac{1}{2}
        \thickness{S}
        \geq
        \aspectratio{S}
        \geq
        (n+1)^{-1} \thickness{S}
        ,
        \qquad 
        \Cjung{n} \, \radiusratio{S}
        \geq
        \aspectratio{S}
        \geq
        \frac{1}{2} \radiusratio{S}
        ,
        \qquad 
        \thickness{S}
        \geq
        \radiusratio{S}
        \geq
        \frac{\Cjung{n}^{-1}}{n+1} \, \thickness{S}
        .
    \end{gather*}
\end{proposition}

It is clear that $\volume{S}$ is bounded from below by the volume of its inscribed ball with radius $\inradius{S}$ and from above by the volume of the enclosing ball with radius $\outradius{S}$. %
Proposition~\ref{proposition:extremal-properties} gives the tighter bounds 
\begin{gather}
    \Cinradius{n} \inradius{S}^{n} \leq \volume{S} \leq \Coutradius{n} \outradius{S}^{n}. \label{math:inradius-volume-outradius}
\end{gather}
Using~\eqref{math:inradius-via-heights-formula} and Proposition~\ref{proposition:extremal-properties} on the codimension-one faces, 
we obtain the inequality 
\begin{align*} 
    \volume{S} 
    = 
    \inradius{S} 
    \sum_{i=0}^{n} 
    \frac{ \Area{i} }{n} 
    \leq 
    \frac{n+1}{n} 
    \Cdiameter{n-1} 
    \inradius{S} 
    \diameter{S}^{n-1} 
    .
\end{align*}
The height of any vertex is bounded by the same vertex's height in any proper subsimplex containing it. 
Therefore, using~\eqref{math:volume-from-height-and-face} recursively leads to 
\begin{align*}
    \frac{1}{n!} \minheight{S}^{n} \leq \volume{S} \leq \frac{\Cdiameter{n-1}}{n} \minheight{S} \diameter{S}^{n-1}
    .
\end{align*}
We state our second set of shape measure comparisons. 
\begin{proposition}
    We have 
    \begin{gather*}
        \frac{n+1}{n} \Cdiameter{n-1}
        \aspectratio{S}
        \geq
        \fullness{S}
        \geq
        \Cinradius{n}
        \aspectratio{S}^{n}
        ,
        \qquad 
        \frac{\Cdiameter{n-1}}{n}
        \thickness{S}
        \geq
        \fullness{S}
        \geq
        \frac{1}{n!} \, \thickness{S}^{n}
        ,
        \qquad
        \frac{n+1}{n} \Cdiameter{n-1} \Cjung{n}
        \radiusratio{S}
        \geq
        \fullness{S}
        \geq
        \frac{\Cinradius{n}}{2^{n}}
        \radiusratio{S}^{n}
        .
    \end{gather*}
\end{proposition}

\begin{remark}
    Let us illustrate these inequalities in low dimensions. If $n=2$ and $S$ is a triangle, then the preceding inequalities become
    \begin{align*}
        \tfrac{1}{2} \thickness{S}
        &
        \geq 
        \aspectratio{S}
        \geq 
        \tfrac{1}{3} \thickness{S},
        &
        \tfrac{1}{\sqrt{3}} \radiusratio{S}
        &
        \geq 
        \aspectratio{S}
        \geq 
        \tfrac{1}{2} \radiusratio{S},
        &
        \thickness{S}
        &
        \geq 
        \radiusratio{S}
        \geq 
        \tfrac{1}{\sqrt{3}} \thickness{S},
        \\
        \tfrac{3}{2} \aspectratio{S}
        &
        \geq 
        \fullness{S}
        \geq 
        3\sqrt{3} \aspectratio{S}^{2},
        &
        \tfrac{1}{2} \thickness{S}
        &
        \geq 
        \fullness{S}
        \geq 
        \tfrac{1}{2} \thickness{S}^{2},
        &
        \tfrac{\sqrt{3}}{2} \radiusratio{S}
        &
        \geq 
        \fullness{S}
        \geq 
        \tfrac{3\sqrt{3}}{4} \radiusratio{S}^{2}
        .
    \end{align*}
    If $n = 3$ and $S$ is a tetrahedron, then the preceding inequalities become
    \begin{align*}
        \tfrac{1}{2} \thickness{S}
        &
        \geq 
        \aspectratio{S}
        \geq 
        \tfrac{1}{4} \thickness{S},
        &
        \sqrt{\tfrac{3}{8}} \radiusratio{S}
        &
        \geq 
        \aspectratio{S}
        \geq 
        \tfrac{1}{2} \radiusratio{S},
        &
        \thickness{S}
        &
        \geq 
        \radiusratio{S}
        \geq 
        \tfrac{1}{\sqrt{6}} \thickness{S},
        \\
        \tfrac{1}{\sqrt{3}} \aspectratio{S}
        &
        \geq 
        \fullness{S}
        \geq 
        8\sqrt{3} \aspectratio{S}^{3},
        &
        \tfrac{\sqrt{3}}{12} \thickness{S}
        &
        \geq 
        \fullness{S}
        \geq 
        \tfrac{1}{6} \thickness{S}^{3},
        &
        \tfrac{1}{2\sqrt{2}} \radiusratio{S}
        &
        \geq 
        \fullness{S}
        \geq 
        \sqrt{3} \radiusratio{S}^{3}.
    \end{align*}
\end{remark}

\section{Solid angles as shape measures}\label{section:solidangles}

Solid angles are higher-dimensional generalizations of the classical angle between two vectors.
We discuss solid angles of simplices and how they relate to the fullness.
The main result of this section can be paraphrased as follows: \emph{simplices are well-shaped if and only if they do not have small solid angles}. 

If $\unitu_{0}, \dots, \unitu_{m} \in \unitsphere{n-1}$ are linearly independent, then 
\begin{align*}
    \calA( \unitu_{0}, \dots, \unitu_{m} ) := \unitsphere{n-1} \cap \cone\{ \unitu_{0}, \dots, \unitu_{m} \}
\end{align*}
is the spherical $m$-dimensional simplex with vertices $\unitu_{0}, \dots, \unitu_{m}$. 
The spherical simplex, the cone, and the vertices uniquely identify each other. 

Assume that $S \subset \bbR^{n}$ is an $n$-dimensional simplex with vertices $v_{0}, \dots, v_{n}$. 
The \emph{solid angle} of $S$ at vertex $v_{i}$ is defined as the measure 
\begin{align}
    \label{math:solidangledefinition:sphere-and-cone}
    \solidangle_{i}(S)
    :=
    \Hausdorff^{n-1}\left( \unitsphere{n-1} \cap \cone\{ v_{0} - v_i, \dots, v_{i-1}-v_i, v_{i+1}-v_i, \dots, v_n-v_i \} \right)
    .
\end{align}
In other words, the edges adjacent to the vertex $v_{i}$ define a spherical simplex, 
\begin{align*}
    \calA_{i}
    :=
    \calA\left( 
        \frac{ v_{0  } - v_{i} }{ \| v_{0  } - v_{i} \| },
        \dots,
        \frac{ v_{i-1} - v_{i} }{ \| v_{i-1} - v_{i} \| },
        \frac{ v_{i+1} - v_{i} }{ \| v_{i+1} - v_{i} \| },
        \dots, 
        \frac{ v_{n  } - v_{i} }{ \| v_{n  } - v_{i} \| }
    \right),
\end{align*}
and the solid angle is the measure of that spherical simplex.
Alternatively,
\begin{align}
    \label{math:solidangledefinition:limit}
    \solidangle_{i}(S)
    :=
    \lim_{\epsilon \rightarrow 0} 
    \dfrac 
    {\Hausdorff^{n-1}\left( \Sphere{v_{i}}{\epsilon} \cap S \right)}
    {\epsilon^{n-1}}
    .
\end{align}
It is easy to see that this quotient attains a constant value for $\epsilon > 0$ small enough.

A distinct formulation of the solid angle occassionally is of great utility. 
Suppose that the matrix $M_{i} \in \bbR^{n \times n}$ has a columns the vectors 
\begin{align*}
    v_{0} - v_{i}, \quad \cdots \quad v_{i-1} - v_{i}, \quad v_{i+1} - v_{i}, \quad \cdots \quad v_{n} - v_{i},
\end{align*}
The solid angle at $v_{i}$ satisfies the representation 
\begin{align}\label{math:orthant-integral}
    \solidangle_{i}(S)
    =
    \frac{ \unitspherearea{n-1} }{ \pi^{\frac n 2} }
    \int_{ M_{i} \cdot \bbR_{+}^{n} }
    \exp( -\| x \|^{2} ) 
    dx 
    =
    \lvert \det(M_{i}) \rvert
    \frac{ \unitspherearea{n-1} }{ \pi^{\frac n 2} }
    \int_{ \bbR_{+}^{n} }
    \exp( -\| M_{i} x \|^{2} ) 
    dx 
    .
\end{align}
as an integral over the positive orthant.

It is a classical result that the sum of the angles within a triangle equals $\pi$. 
The generalization to higher dimension takes a weaker form. 
The sum of the solid angles in a tetrahedron is at most $2\pi$ and can be arbitrarily close to zero~\cite{gaddum1952sums}.
More generally, the sum of the solid angles in an $n$-dimensional simplex is at most $\unitspherearea{n-1} / 2$.
This extension to higher dimensions can be embedded in the discussion of convex polytopes~\cite[(24)]{perles1967angle} and random polytopes~\cite[Theorem~3.10]{Kabluchko2021recursive} but also admits an elementary proof that is inspired by the classical planar case. We include it for completeness.

\begin{theorem}\label{theorem:sum-of-solid-angles}
    Let $S$ be an $n$-dimensional simplex with vertices $v_0, \dots, v_{n} \in \bbR^{n}$, for $n \geq 2$. Then
    \begin{align}\label{math:sum-of-solid-angles}
        \sum_{i=0}^{n} \solidangle_{i}(S) \leq \frac{\unitspherearea{n-1}}{2}.
    \end{align}
    In particular, the minimum solid angle of $S$ is at most $\unitspherearea{n-1} / (2n+2)$. 
\end{theorem}
\begin{proof}
    Without loss of generality, after a translation, $v_{0} = 0$.
    Suppose that $x \in \bbR^{n}$. 
    There exist unique $\lambda_1,\dots,\lambda_{n}, \mu_{0},\dots,\mu_{n-1} \in \bbR$ such that 
    \begin{align*}
        x = \sum_{i=1}^{n} \lambda_{i} ( v_{i} - v_{0} ) = \sum_{i=0}^{n-1} \mu_{i} ( v_{i} - v_{n} ).
    \end{align*}
    We rearrange this equation,
    \begin{align*}
        (-\tsum_{i=1}^{n} \lambda_{i}) v_{0}
        +
        \sum_{i=1}^{n-1} \lambda_{i} v_{i}
        +
        \lambda_{n} v_{n}
        =
        \mu_{0} v_{0}
        +
        \sum_{i=1}^{n-1} \mu_{i} v_{i}
        +
        (-\tsum_{i=0}^{n-1} \mu_{i}) v_{n}
        ,
    \end{align*}
    and notice that $\lambda_{i} = \mu_{i}$ for $1 \leq i \leq n-1$.
    In addition, $\lambda_{n} = -\tsum_{i=0}^{n-1} \mu_{i}$. 
    Consider the open cones 
    \begin{gather*}
        C_{0}^{+} := \sum_{i=1}^{n}   \bbR^{+} \cdot ( v_{i} - v_{0} ),
        \qquad 
        C_{0}^{-} := - C_{0}^{+},
        \qquad 
        C_{n}^{+} := \sum_{i=0}^{n-1} \bbR^{+} \cdot ( v_{i} - v_{n} ),
        \qquad 
        C_{n}^{-} := - C_{n}^{+}.
    \end{gather*}
    Clearly, $C_{0}^{+} \cap C_{0}^{-} = C_{n}^{+} \cap C_{n}^{-} = \emptyset$. 
    If $x \in C_{0}^{+}$, then positivity of $\lambda_{1}, \dots, \lambda_{n-1}$ shows that $x \notin C_{n}^{-}$,
    and positivity of $\lambda_{n}$ shows that $x \notin C_{n}^{+}$. 
    Likewise, $C_{n}^{+}$ is disjoint from $C_{0}^{-}$. 
    We conclude that the spherical simplices defined by the $n+1$ vertices, and their reflections across the origin,
    \begin{align*}
        \calA_{0}, \calA_{1}, \dots, \calA_{n}, -\calA_{0}, -\calA_{1}, \dots, -\calA_{n},
    \end{align*}
    are essentially disjoint within the unit sphere of dimension $n-1$. 
    The inequality~\eqref{math:sum-of-solid-angles} is now evident. 
\end{proof}

The minimum solid angle 
\begin{align*}
    \minsolidangle{S} := \min\{ \solidangle_{0}(S), \solidangle_{1}(S), \dots, \solidangle_{n}(S) \} 
\end{align*}
of the simplex $S$ is a candidate for a shape measure. 
The simplex $S$ is contained in the volume suspended by any vertex and its opposite spherical simplex, scaled by the radius $\diameter{S}$. 
Hence, 
\begin{align*}
    \volume{S} 
    \leq 
    \frac{\minsolidangle{S}}{\unitspherearea{n-1}} \unitballvolume{n} \diameter{S}^{n} 
    = 
    \frac{ \diameter{S}^{n} }{n} \minsolidangle{S}
    .
\end{align*}
Therefore,
\begin{align*}
    \fullness{S} \leq \frac{1}{n} \minsolidangle{S}.
\end{align*}
This means that well-shaped simplices do not have small solid angles. 
The purpose of this manuscript is to establish the converse implication.

\begin{example}
    Let $\epsilon \in (0,1)$. Consider a simplex $S_{\epsilon}$ with vertices 
    \begin{align*}
        v_{0} = 0, 
        \quad 
        v_{1} = e_{1}, 
        \quad 
        v_{2} = e_{1} + \epsilon e_{2}, 
        \quad \dots \quad 
        v_{n} = e_{1} + \epsilon e_{n}.
    \end{align*}
    Obviously, $\diameter{S_{\epsilon}} = \sqrt{ 1 + \epsilon^{2} }$ and $\volume{S_{\epsilon}} = \epsilon^{n-1} / n!$,
    and therefore 
    \begin{align*}
        \fullness{S_{\epsilon}} 
        = 
        \volume{S_{\epsilon}} \diameter{S_{\epsilon}}^{-n}
        = 
        \frac{\epsilon^{n-1}}{n! (1+\epsilon^{2})^{\frac{n}{2}} }
        .
    \end{align*}
    Always, $\solidangle_{1} = 2^{-n} \unitspherearea{n-1}$, and the solid angles at $v_2,\dots,v_{n}$ monotonically increase towards a constant value as $\epsilon$ decreases. 
    For the solid angle at $v_{0}$, consider the matrix 
    \begin{align*}
        M_{0}
        :=
        \begin{pmatrix}
            1 &        1 & \dots  & 1
            \\
              & \epsilon &        &           
            \\
              &          & \ddots &           
            \\
              &          &        & \epsilon 
        \end{pmatrix}
        .
    \end{align*}
    If $x \in \bbR^{n}$ with positive coordinates, then one easily verifies 
    \begin{align*}
        \| M_{0} x \|^{2} \leq (n+\epsilon^{2}) \|x\|^{2}
        ,
        \qquad 
        \| M_{0} x \|^{2} \geq \|x\|^{2}
        .
    \end{align*}
    Applying the lower and upper bound in the representation~\eqref{math:orthant-integral} of the solid angles as orthant integral and performing a substitution, the following inequalities emerge:
    \begin{align*}
        2^{-n} \unitspherearea{n-1}
        \frac{ \epsilon^{n-1} }{ n! (n+\epsilon)^{\frac n 2} } 
        \leq 
        \solidangle_{0}
        \leq 
        2^{-n} \unitspherearea{n-1}
        \frac{ \epsilon^{n-1} }{ n! 2^{n} } 
        .
    \end{align*}
    This shows that any upper bound for the fullness in terms of the minimum solid angle, 
    in two or more dimensions, must be at least linear. 
\end{example}

Let $\halfspace_{n}$ be the northern halfspace, and let $\bbE_{n-1} \subset \bbR^{n}$ be the tangent plane at the north pole,
\begin{align*}
    \halfspace_{n} := \{ (x_{1},\dots,x_{n}) \in \bbR^{n} \suchthat x_{n} > 0 \},
    \qquad 
    \bbE_{n-1} := \{ x \in \bbR^{n} \suchthat x_{n} = 1 \}.
\end{align*}
The gnomonic projection is 
\begin{align*}
    \Gnomonic : \unitsphere{n-1} \cap \halfspace_{n} \to \bbE_{n-1}, \quad (x_{1},\dots,x_{n-1},x_{n}) \mapsto (x_{1}/x_{n},\dots,x_{n-1}/x_{n},1).
\end{align*}
Its inverse takes the form 
\begin{align*}
    \Gnomonic^{-1} : \bbE_{n-1} \to \unitsphere{n-1} \cap \halfspace_{n}, \quad y \mapsto \frac{y}{\|y\|}.
\end{align*}
The gnomonic projection $\Gnomonic$ maps spherical simplices in the northern hemisphere $\unitsphere{n-1} \cap \halfspace_{n}$ to simplices within the tangential plane $\bbE_{n-1}$. 

\begin{lemma}
    Let $\unitu_{0}, \dots, \unitu_{m} \in \unitsphere{n-1} \cap \halfspace_{n}$ be linearly independent unit vectors in the northern hemisphere. 
    Write $u_{i} := \Gnomonic( \unitu_{i} )$ for $0 \leq i \leq m$.
    Then $\convex\{ u_{0}, \dots, u_{m} \}$ is an $m$-dimensional simplex with vertices $u_{0},\dots,u_m$ and 
    \begin{align*}
        \Gnomonic( \calA( \unitu_{0}, \dots, \unitu_{m} ) ) = \convex\{ u_{0}, \dots, u_{m} \}.
    \end{align*}
\end{lemma}
\begin{proof}
    Write $C := \cone\{ \unitu_{0}, \dots, \unitu_{m} \}$ and $\calA := \calA( \unitu_{0}, \dots, \unitu_{m} )$.
    We know that $K := \bbE_{n-1} \cap C$ is a simplex with vertices $u_{0}, \dots, u_{m}$.
    By definition, $\calA = \unitsphere{n-1} \cap C$. 
    Since $\Gnomonic$ preserves half-lines starting at the origin, $\Gnomonic(\calA) = K$ is evident. 
\end{proof}

\begin{lemma} 
    For all $y \in \bbE_{n-1}$ we have 
    \begin{align*}
        \det \left( \Jacobian \Gnomonic^{-t}_{|y} \Jacobian \Gnomonic^{-1}_{|y} \right)
        = 
        \|y\|^{-2n}
        .
    \end{align*}
\end{lemma}
\begin{proof}
    Recall that the normalization mapping $\Normalization : \bbR^{n} \setminus \{0\} \to \unitsphere{n-1}$ from the punctured Euclidean space onto the unit sphere, which is defined by $\Normalization(y) := y / \|y\|$, 
    has a Jacobian of the form 
    \begin{align*}
        \Jacobian \Normalization[y]
        &
        =
        \frac{1}{ \|y\| } \Id_{n\times n}
        -
        \frac{ y \otimes y^{t} }{\|y\|^{3}}
        .
    \end{align*}
    The Jacobian of $\Gnomonic^{-1}$ at any $y \in \bbE_{n-1}$ in tangential direction $v \in \bbR^{n}$, 
    where we write $y = (y_{1},\dots,y_{n-1},1)$ and $v = (v_{1},\dots,v_{n-1},0)$ and abbreviate $z := (y_{1},\dots,y_{n-1})$, takes the following form:
    \begin{align*}
        \Jacobian \Gnomonic^{-1}[y]
        &
        =
        ( \|z\|^2 + 1 )^{-\frac 1 2} \Id_{n \times (n-1)}
        -
        ( \|z\|^{2} + 1 )^{-\frac{3}{2}} ( (z,0)+e_{n}) \otimes z^{t}. 
    \end{align*}
    We abbreviate $s=\|z\|^{2}+1$ and find
    \begin{align*}
        &
        \bigl(\Jacobian \Gnomonic^{-1}[y]\bigr)^{t}
        \bigl(\Jacobian \Gnomonic^{-1}[y]\bigr)
        =
        \frac{1}{s}\Id_{n-1}
        -
        \frac{2}{s^{2}}\,z\otimes z^{t}
        +
        \frac{\|(z,1)\|^{2}}{s^{3}}\,z\otimes z^{t}
        =
        \frac{1}{s}
        \left(
            \Id_{n-1}
            -
            \frac{z\otimes z^{t}}{s}
        \right)
        .
    \end{align*}
    Therefore, the matrix determinant lemma applies and gives 
    \begin{align*}
        &
        \det\left(
            \Id_{n-1}
            -
            \frac{z\otimes z^{t}}{s}
        \right)
        =
        1-\frac{z^{t}z}{s}
        =
        1-\frac{\|z\|^{2}}{\|z\|^{2}+1}
        =
        \frac{1}{s}
        .
    \end{align*}
    The desired determinant formula follows. 
\end{proof}

\begin{remark}\label{remark:one-dimensional-gnomonic} 
    The gnomonic projection of the one-dimensional northern hemisphere, the upper semicircle, elegantly relates to trigonometric identities.
    Let $x = (x_1,x_2) \in \unitsphere{1} \cap \halfspace_{2}$ be in the northern hemisphere, and let $(z,1) = \Gnomonic(x)$. 
    Let $\alpha = \arctan(z) \in (-\pi/2,\pi/2)$. 
    Then $x_1 = \sin(\alpha)$ and $x_2 = \cos(\alpha)$. 
    At the same time, 
    $x_1 = \sin(\arctan(z)) = z / \sqrt{1+z^{2}}$ and $x_2 = \cos(\arctan(z)) = 1 / \sqrt{1+z^{2}}$.
    From $z = x_1 / x_2$ and $1 = x_{1}^{2} + x_{2}^{2}$, converse identities
    $z = \tan( \arcsin(x_1 ) ) = x_1 / \sqrt{ 1 - x_1^{2} }$ and $|z| = \tan( \arccos(x_2) ) = \sqrt{ 1 - x_2^{2} } / x_2$
    are derived. 
\end{remark}

The area of spherical simplices, and thus solid angles, can be formulated as a Euclidean integral. 
Suppose that $\unitu_{1}, \dots, \unitu_{n} \in \unitsphere{n-1} \cap \halfspace_{n}$ are linearly independent unit vectors in the northern hemisphere. 
Then
\begin{align*}
    \calA := \unitsphere{n-1} \cap \cone \{ \unitu_{1}, \dots, \unitu_{n} \}
\end{align*}
is a spherical simplex that is mapped along the gnomonic projection to the simplex 
\begin{align*}
    K := \Gnomonic(\calA) = \convex \{ \Gnomonic(\unitu_{1}), \dots, \Gnomonic(\unitu_{n}) \}.
\end{align*}
The area of the spherical simplex admits the representation as a weighted Euclidean integral 
\begin{align*}
    \Hausdorff^{n-1}(\calA) 
    =
    \int_{K} \|y\|^{-n} d\Hausdorff^{n-1}
    .
\end{align*}
Notice that the integrand is non-singular and smooth because $K \subset \bbE_{n-1}$. 
This reformulation of solid angles is central to our analysis of solid angles. 

\section{The minimum solid angle}\label{section:main-result}

We first address the special case of a tetrahedron, $n=3$, and then the case of simplices in any dimension $n$. 

\begin{theorem}\label{theorem:tetrahedral-main-result}
    Let $S \subset \bbR^{3}$ be a tetrahedron. We have the estimate 
    \begin{align}\label{math:tetrahedral-main-result}
        C_{\mathrm{tet}}^{-1}
        \minsolidangle{S}^{2}
        \leq 
        \fullness{S}
        \leq 
        \tfrac{1}{3} \minsolidangle{S}
        ,
        \qquad 
        C_{\mathrm{tet}}
        :=
        12 + 3\pi 
        .
    \end{align}
\end{theorem}
\begin{proof}
    Let $S$ be a tetrahedron with vertices $v_{0}, v_{1}, v_{2}, v_{3} \in \bbR^{3}$.
    It suffices to consider the case $\diameter{S} = 1$ because the rest follows by a scaling argument. 
    Without loss of generality, 
    after rotation, translation, and relabeling, 
    $v_{0} = 0$ and $v_{3} = e_{3}$.
    Since the edge from $v_{0}$ to $v_{3}$ is the diameter of $S$, the remaining vertices have the form
    \begin{align*}
        v_1 = (f_1, \theta_1 ),
        \qquad
        v_2 = (f_2, \theta_2 ),
    \end{align*}
    where $f_1, f_2 \in \bbR^{2}$ and $\theta_1, \theta_2 \in (0,1)$, and $f_1, f_2$ are linearly independent.
    These satisfy the inequalities 
    \begin{align*}
        \| f_1 \|^2 + \theta_1^2 \leq 1,
        \qquad
        \| f_2 \|^2 + \theta_2^2 \leq 1,
        \qquad
        \| f_1 \|^2 + (1-\theta_1)^2 \leq 1,
        \qquad
        \| f_2 \|^2 + (1-\theta_2)^2 \leq 1.
    \end{align*}
    The vertices $v_1, v_2, v_{3}$ lie in the northern half-space $\halfspace_{3}$.
    Let $\unitv_1, \unitv_2$ be the unit vectors in the directions of $v_1, v_2$, respectively.
    Their gnomonic projections onto the northern tangential plane $\bbE_{2}$ are
    \begin{align*}
        \Gnomonic(\unitv_1) = \left( f_1 / \theta_1, 1 \right),
        \qquad
        \Gnomonic(\unitv_2) = \left( f_2 / \theta_2, 1 \right)
        .
    \end{align*}
    Taking the gnomonic projection of the spherical triangle opposite to $v_{0}$ onto $\bbE_{2}$ and omitting the third coordinate gives the triangle 
    \begin{align*}
        K_{0} = \convex\left\{ 0, \frac{f_1}{\theta_1}, \frac{f_2}{\theta_2} \right\}
        \subset \bbR^{2}.
    \end{align*}
    Therefore, the solid angle $\solidangle_{0} := \solidangle_{0}(S)$ equals the planar integral 
    \begin{align*}
        \solidangle_{0}
        =
        \iint_{K_{0}} ( 1 + \|z\|^2)^{-\frac{3}{2}} dz
        .   
    \end{align*}
    A fully analogous argument rewrites the solid angle $\solidangle_{3} := \solidangle_{3}(S)$ as the planar integral
    \begin{align*}
        \solidangle_{3}
        =
        \iint_{K_{3}} ( 1 + \|z\|^2)^{-\frac{3}{2}} dz
        ,
        \qquad 
        K_{3} = \convex\left\{ 0, \frac{f_1}{1 - \theta_1}, \frac{f_2}{1 - \theta_2} \right\}
        .
    \end{align*}
    We define the new triangle $F = \convex\{ 0, f_1, f_2 \}$. 
    Let $\alpha \in (0,\pi)$ be the angle between $f_{1}$ and $f_{2}$.
    By Cavalieri's principle, $3 \volume{S} = \volume{F}$. 
    By construction, the measures satisfy 
    \begin{gather*}
        \volume{K_{0}} = \theta_1^{-1} \theta_2^{-1} \volume{F},
        \qquad
        \volume{K_{3}} = (1 - \theta_1)^{-1} (1 - \theta_2)^{-1} \volume{F},
        \qquad 
        3 \volume{S} = \volume{F} = \frac{1}{2} \sin(\alpha) b_1 b_2, 
        \\ 
        b_1 := \| f_1 \|,
        \qquad 
        b_2 := \| f_2 \|.
    \end{gather*}
    We can now estimate the solid angles. First, we notice the simple estimates 
    \begin{align*}
        \solidangle_{0} \leq \volume{K_{0}},
        \qquad
        \solidangle_{3} \leq \volume{K_{3}}
        .
    \end{align*}
    Since $K_{0}$ and $K_{3}$ are triangles with the origin as a vertex, 
    they lie in a radial sector of $\bbR^{2}$ with angle $\alpha$.
    Thus,
    \begin{align*}
        \solidangle_{0} 
        &
        \leq
        \frac{\alpha}{2\pi} 
        \iint_{\bbR^{2}} ( 1 + \|z\|^2)^{-\frac{3}{2}} dz 
        =
        \alpha 
        \int_{0}^{\infty} ( 1 + r^2 )^{-\frac{3}{2}} r dr 
        =
        \alpha 
        \int_{0}^{\infty} \left( -( 1 + r^2 )^{-\frac{1}{2}} \right)' dr 
        =
        \alpha 
        .
    \end{align*}
    The same argument applies to $K_{3}$. Thus, 
    $\solidangle_{0} \leq \alpha$, and $\solidangle_{3} \leq \alpha$, 
    which yields the product estimates 
    \begin{align*}
        \minsolidangle{S}^{2}
        \leq 
        \solidangle_{0} \solidangle_{3}
        \leq 
        \theta_{1}^{-1} \theta_{2}^{-1} 
        \volume{F} \alpha
        ,
        \qquad 
        \minsolidangle{S}^{2}
        \leq 
        \solidangle_{0} \solidangle_{3}
        \leq 
        (1 - \theta_1)^{-1} (1 - \theta_2)^{-1}
        \volume{F} \alpha
        .
    \end{align*}
    At the same time, Theorem~\ref{theorem:sum-of-solid-angles} states that the minimum solid angle of a tetrahedron is at most $\pi/2$, hence
    \begin{align*}
        \minsolidangle{S}^{2}
        \leq
        \frac{\pi}{2} \solidangle_{0}
        \leq 
        \theta_{1}^{-1} \theta_{2}^{-1} 
        \frac{\pi}{2} \volume{F},
        \qquad 
        \minsolidangle{S}^{2}  
        \leq
        \frac{\pi}{2} \solidangle_{3}
        \leq 
        (1 - \theta_1)^{-1} (1 - \theta_2)^{-1}
        \frac{\pi}{2} \volume{F}
        .
    \end{align*}
    The inclusion $K_{0} \subseteq [0,1]\theta_1^{-1} f_1 + \bbR_{0}^{+} \theta_2^{-1} f_2$ gives,
    after an orthogonal transformation that aligns $f_{2}$ with the second coordinate direction, 
    \begin{align*}
        \solidangle_{0}
        \leq
        I( \theta_{1}^{-1} b_{1}, \alpha )
        ,
        \qquad 
        I( \theta_{1}^{-1} b_{1}, \alpha )
        :=
        \int_{0}^{ \theta_{1}^{-1} b_{1} \sin(\alpha) }
        \int_{ x  \cot(\alpha) }^{\infty}
        ( 1 + x^2 + y^{2} )^{-\frac{3}{2}}
        dy dx
        .
    \end{align*}
    Abbreviate $B =  \cot(\alpha) x$. 
    We substitute $y = \sqrt{1+x^{2}} \tan(\varphi)$ and find 
    \begin{align*}
        \int_{ B }^{\infty}
        ( 1 + x^2 + y^{2} )^{-\frac{3}{2}} dy 
        &
        =
        \int_{ \arctan( B / \sqrt{1+x^{2}} ) }^{\frac{\pi}{2}}
        \frac{ \sec(\varphi)^{2} }{ (1+x^{2}) ( 1 + \tan(\varphi)^{2} )^{\frac{3}{2}} } d\varphi 
        =
        \frac{1}{ 1+x^{2} }
        \int_{ \arctan( B / \sqrt{1+x^{2}} ) }^{\frac{\pi}{2}}
        \frac{ \sec(\varphi)^{2} }{  \sec(\varphi)^{3} } d\varphi 
        \\&
        =
        \frac{1}{ 1+x^{2} }
        \int_{ \arctan( B / \sqrt{1+x^{2}} ) }^{\frac{\pi}{2}}
        \cos(\varphi) d\varphi 
        =
        \frac{ 1 - \sin( \arctan( B / \sqrt{1+x^{2}} ) ) }{ 1+x^{2} }
        \\
        &
        =
        \frac{ 1 }{ 1+x^{2} }
        -
        \frac{ \sin\arctan( \cot(\alpha) \tfrac{ x }{ \sqrt{1+x^{2}} } ) }{ 1+x^{2} }
        .
    \end{align*}
    As in Remark~\ref{remark:one-dimensional-gnomonic}, one can simplify 
    \begin{align*}
        &
        \sin\arctan( \cot(\alpha) \tfrac{ x }{ \sqrt{1+x^{2}} } )
        =
        \frac{ \cot(\alpha) x }{ \sqrt{ 1 + ( 1 + \cot^{2}(\alpha) ) x^{2} } }
        =
        ( \sin^{2}(\alpha) + x^{2} )^{-\frac 1 2} \cos(\alpha) x 
        .
    \end{align*}
    We bound the integral of this expression in $x$ from $0$ to $\theta_{1}^{-1} b_{1} \sin(\alpha)$. 
    For any $D > 0$, we notice 
    \begin{align*}
        \partial_D I( D, \alpha )
        &
        :=
        \partial_D
        \int_{0}^{ D \sin(\alpha) }
        \frac{ 1 }{ 1+x^{2} }
        -
        \frac{ \cos(\alpha) x }{ (1+x^2) ( \sin^{2}(\alpha) + x^{2} )^{ \frac 1 2} }
        dx
        \\&\qquad 
        =
        \frac{ \sin(\alpha) }{ 1 + D^{2} \sin(\alpha)^{2} }
        -
        \frac{ \cos(\alpha) D \sin(\alpha)^{2} }{ ( 1 + D^{2} \sin(\alpha)^{2} ) ( \sin^{2}(\alpha) + D^{2} \sin(\alpha)^{2} )^{ \frac 1 2} }
        \\&\qquad 
        =
        \frac{ \sin(\alpha) }{ 1 + D^{2} \sin(\alpha)^{2} }
        -
        \frac{ \cos(\alpha) \sin(\alpha) D }{ ( 1 + D^{2} \sin(\alpha)^{2} ) ( 1 + D^{2} )^{ \frac 1 2} }
        \\&\qquad 
        =
        \frac{ \sin(\alpha) }{ 1 + D^{2} \sin(\alpha)^{2} }
        \left( 1 - \cos(\alpha) \sin(\arctan(D)) \right)
        .
    \end{align*}
    The integral $I( \theta_{1}^{-1} b_{1}, \alpha )$ is increasing in $\alpha$ up to $\pi/2$.
    If $\alpha = \pi/2$, then the bound $\partial_D I( D, \alpha ) \leq 1$ leads to $I( \theta_{1}^{-1} b_{1}, \alpha ) \leq \theta_{1}^{-1} b_{1}$,
    so we consider now the case where $0 < \alpha < \pi/2$ or $\pi/2 < \alpha < \pi$. 
    \newcommand{\X}{\varphi}
    We write $\X := \sin(\arctan(D)) \in (0,1)$, so that $D = \X / \sqrt{ 1 - \X^2 }$. 
    One observes
    \begin{align*}
        &
        \frac{ \sin(\alpha) }{ 1 + D^{2} \sin(\alpha)^{2} }
        \left( 1 - \cos(\alpha) \sin(\arctan(D)) \right)
        \\&
        =
        \frac{ \sin(\alpha) ( 1 - \X^{2} ) \left( 1 - \cos(\alpha) \X \right) }{ 1 - \X^{2} + ( 1 - \X^{2} ) D^{2} \sin(\alpha)^{2} }
        \\&
        =
        \frac{ \sin(\alpha) ( 1 - \X^{2} ) \left( 1 - \cos(\alpha) \X \right) }{ 1 - \X^{2} + \sin(\alpha)^{2} \X^{2} }
        =
        \frac{ \sin(\alpha) ( 1 - \X^{2} ) \left( 1 - \cos(\alpha) \X \right) }{ 1 - \cos(\alpha)^{2} \X^{2} }
        =
        \frac{ \sin(\alpha) ( 1 - \X^{2} )}{ 1 + \cos(\alpha) \X }
        =
        \tan(\alpha) \frac{ 1 - \X^{2} }{ \sec(\alpha) + \X }
        .
    \end{align*} 
    The last quotient as a function in $\X$ vanishes at $-1$ and $+1$ and is positive over the interval $(-1,1)$. 
    Therefore, we seek its critical points at the solutions of the quadratic equation 
    \begin{align*}
        \X^{2} + 2 \X \sec(\alpha) + 1 = 0.
    \end{align*}
    Its two roots are $\X = - \sec(\alpha) \pm \sqrt{ \sec(\alpha)^{2} - 1 }$.
    We only take the critical point within $(-1,1)$. 
    If $0 < \alpha < \pi/2$, then this is $\X = - \sec(\alpha) + \sqrt{ \sec^{2}(\alpha) - 1 }$. 
    If $\pi/2 < \alpha < \pi$, then this is $\X = - \sec(\alpha) - \sqrt{ \sec^{2}(\alpha) - 1 }$. 
    In either case, $\X = \tan(\alpha) - \sec(\alpha)$.
    Therefore, 
    \begin{align*}
        \tan(\alpha) \frac{ 1 - \X^{2} }{ \sec(\alpha) + \X }
        &
        =
        1 - \left( \sec(\alpha) - \tan(\alpha) \right)^{2}
        \\&
        =
        1 - \left( \frac{ 1 - \sin(\alpha) }{ \cos(\alpha) } \right)^{2}
        =
        \frac{ 2 \sin(\alpha) - 2\sin(\alpha)^{2} }{ \cos(\alpha)^{2} } 
        =
        2 \sin(\alpha) \frac{ 1 - \sin(\alpha) }{ 1 - \sin(\alpha)^{2} } 
        =
        \frac{ 2 \sin(\alpha) }{ 1 + \sin(\alpha) } 
        .
    \end{align*}
    We integrate $\partial_{D} I(D,\alpha)$ from $0$ to $b_{1} / \theta_{1}$, 
    thus bounding $I(b_{1}/\theta_{1},\alpha)$. 
    For all $0 < \alpha < \pi$, we have verified 
    \begin{align*}
        I( \theta_{1}^{-1} b_{1}, \alpha )
        \leq 
        \theta_{1}^{-1} b_{1}
        \frac{ 2 \sin(\alpha) }{ 1 + \sin(\alpha) } 
        .
    \end{align*}
    This leads to the following upper estimate for $\solidangle_{0}$, another upper estimate based on the second vertex, and analogous upper bounds for the opposite solid angle $\solidangle_{3}$:
    \begin{align*}
        \solidangle_{0}
        \leq 
        \frac{
            2 b_{1} \theta_{1}^{-1} \sin(\alpha) 
        }{
            1 + \sin(\alpha)
        }
        ,
        \quad 
        \solidangle_{0}
        \leq 
        \frac{
            2 b_{2} \theta_{2}^{-1} \sin(\alpha) 
        }{
            1 + \sin(\alpha)
        }
        ,
        \quad 
        \solidangle_{3}
        \leq 
        \frac{
            2 b_{1} (1-\theta_{1})^{-1} \sin(\alpha) 
        }{
            1 + \sin(\alpha)
        }
        ,
        \quad 
        \solidangle_{3}
        \leq 
        \frac{
            2 b_{2} (1-\theta_{2})^{-1} \sin(\alpha) 
        }{
            1 + \sin(\alpha)
        }
        .
    \end{align*}
    Along these lines of reasoning, the following product estimates are established:
    \begin{align*}
        \solidangle_{0} \solidangle_{3} 
        \leq 
        \theta_{1}^{-1} (1-\theta_{2})^{-1}
        b_{1} b_{2} 
        \sin(\alpha)
        \frac{
            4  \sin(\alpha) 
        }{
            ( 1 + \sin(\alpha) )^{2}
        }
        \leq 
        \theta_{1}^{-1} (1-\theta_{2})^{-1}
        2\volume{F}
        ,
        \qquad 
        \solidangle_{0} \solidangle_{3} 
        \leq 
        (1-\theta_{1})^{-1} \theta_{2}^{-1}
        2\volume{F}
        .
    \end{align*}
    Here, we have used that $4 \sin(\alpha) \leq ( 1 + \sin(\alpha) )^{2}$ because $2 \sqrt{\sin(\alpha)} \leq 1 + \sin(\alpha)$.
    This leads to the following sum:
    \begin{align*}
        \minsolidangle{S}^{2}
        &
        =
        \minsolidangle{S}^{2}
        \left( \theta_{1} \theta_{2} + \theta_{1} (1 - \theta_2) + (1 - \theta_1) \theta_{2} + (1 - \theta_1) (1 - \theta_2)  \right)
        \leq 
        2\left( \tfrac{\pi}{2} + 2 \right) \volume{F}
        =
        ( 12 + 3\pi ) \volume{S}
        .
    \end{align*}
    Obviously, this proves~\eqref{math:tetrahedral-main-result} in the case $\diameter{S} = 1$.     
    Finally, the assumption $\diameter{S} = 1$ can be omitted by a scaling argument. 
    This proves the desired estimate.     
\end{proof}

\begin{remark}\label{remark:tetrahedral-motivation}
    We remark on the geometric intuition in this proof. 
    First, $\volume{F}$ is proportional to $\volume{S}$ by Cavalieri's principle. 
    If $v_1$ and $v_2$ are closer to $v_{3}$, say, $\theta_1, \theta_2 > \tfrac{1}{2}$, 
    then the solid angle $\solidangle_{0} \leq \volume{K_{0}} \leq 4 \volume{F}$, since the area distortion of the gnomonic projection is bounded in this regime. 
    The solid angle $\solidangle_{3}$ is at most $2\pi$ by Theorem~\ref{theorem:sum-of-solid-angles}. 
    If instead $v_1$ and $v_2$ are closer to $v_{0}$, say, $\theta_1, \theta_2 \leq \tfrac{1}{2}$,
    we reverse the roles of which solid angle is controlled and which is estimated only uniformly. 
    In either case, the product $\solidangle_{0} \solidangle_{3}$ is controlled.
    
    It remains to address the mixed case. 
    If $\theta_{1} > \tfrac{1}{2}$ but $\theta_{2} \leq \tfrac{1}{2}$, then $K_{0}$ is contained within the strip $[ 0, \theta_{1}^{-1} \|f_1\| \sin(\alpha) ] \times \bbR$, after an orthogonal change of coordinates.
    The solid angle $\solidangle_{0}$ is bounded by a fixed multiple of the width $\theta_{1}^{-1} \|f_1\| \sin(\alpha) \leq 2 \|f_1\| \sin(\alpha)$.
    An analogous estimate applies to the opposite solid angle $\solidangle_{3}$. 
    Each of these solid angles is controlled by one leg of the triangle $F$,
    and hence their product is controlled by both legs. 
    This allows us to use $2 \volume{F} = \|f_1\| \|f_2\| \sin(\alpha)$.
    
    Hence, after picking in accordance to whether each of $v_1$ and $v_2$ is closer to $v_0$ or $v_{3}$, 
    we conduct a case distinction and bound $\solidangle_{0} \solidangle_{3}$ by a multiple of $\volume{S}$.
    The proof of Theorem~\ref{theorem:tetrahedral-main-result} refines this idea:
    instead of picking the worst possible bound,
    we take a weighted average of these estimates.
\end{remark}

\begin{example}\label{example:optimality-of-tetrahedral-quadratic-exponent}
    We show that the quadratic exponent in Theorem~\ref{theorem:tetrahedral-main-result} is optimal and cannot be lowered.
    Let $\epsilon \in (0,1/2)$.
    Inspired by the sliver example in Figure~\ref{figure:simplex-degeneracies}, we consider the tetrahedron $S$ with vertices
    \begin{align*}
        v_{0} = -\epsilon e_{1},
        \quad
        v_{1} = \epsilon e_{1},
        \quad
        v_{2} = e_{2} - \epsilon e_{3},
        \quad
        v_{3} = e_{2} + \epsilon e_{3}.
    \end{align*}
    We have $\diameter{S} = \sqrt{1 + 2\epsilon^{2}}$, and a simple determinant calculation gives $\volume{S} = 2\epsilon^{2}/3$.
    Write $d_{ij} := \|v_i-v_j\|$ for $0 \leq i < j \leq 3$.
    Then
    \begin{align*}
        d_{01} = d_{23} = 2\epsilon,
        \qquad
        d_{02} = d_{03} = d_{12} = d_{13} = \sqrt{1 + 2\epsilon^{2}}.
    \end{align*}
    The solid angles of $S$ are identical because of symmetry.
    A theorem by Liu and Joe~\cite[Theorem~2.2]{liu1994relationship} gives 
    \begin{align*}
        \sin\left( \solidangle_{0}(S) / 2 \right)
        =
        12 \volume{S}
        \cdot
        \prod_{ 1 \leq i < j \leq 3 }
        ( d_{0i} + d_{0j} + d_{ij} )^{-1/2}
        ( d_{0i} + d_{0j} - d_{ij} )^{-1/2}
        .
    \end{align*}
    A direct calculation gives
    \begin{align*}
        \sin\left( \solidangle_{0}(S) / 2 \right)
        &=
        \frac{ \epsilon }{ ( \epsilon + \sqrt{ 1 + 2\epsilon^{2} } ) \sqrt{ 1 + \epsilon^{2} } }
        \in
        (0,1)
        .
    \end{align*}
    Write $k(\epsilon) := (\epsilon+\sqrt{1+2\epsilon^{2}}) \sqrt{1+\epsilon^{2}}$.
    Then
    \begin{align*}
        \solidangle_{0}(S)/2
        =
        \arcsin\left(
            \epsilon / k(\epsilon)
        \right).
    \end{align*}
    Consequently, for every fixed $1 \leq p < \infty$ it holds that 
    \begin{align*}
        \frac{ \minsolidangle{S}^{p} }{ \fullness{S} }
        &=
        \frac{3}{2\epsilon^{2}}
        ( 1 + 2 \epsilon^{2} )^{\frac{3}{2}}
        \left(
            2
            \arcsin\left( \frac{ \epsilon }{ k(\epsilon) } \right)
        \right)^{p}
        .
    \end{align*}
    If $\epsilon$ converges to zero, then $k(\epsilon)$ converges to one,
    and the quotient diverges for every $p < 2$, whereas it converges to the limit $6$ for $p = 2$. 
    Thus, $p=2$ is the lowest possible exponent in the lower bound of \eqref{math:tetrahedral-main-result}.
    Moreover, any admissible constant $C_{\mathrm{tet}}$ in that inequality must satisfy $C_{\mathrm{tet}} \geq 6$.
\end{example}

We analyze the relationship between the fullness and the minimum solid angle of higher dimensional simplices. 

\begin{theorem}\label{theorem:main-result}
    Let $S \subset \bbR^{n}$ be an $n$-dimensional simplex with $n \geq 2$. 
    We have the estimate 
    \begin{align}\label{math:main-result}
        C_{n}^{-1}
        \minsolidangle{S}^{2}
        \leq 
        \fullness{S}
        \leq 
        \tfrac{1}{n} \minsolidangle{S}
        ,
        \qquad 
        C_{n}
        :=
        n 
        \pi^{\frac{ n-1 }{ 2 }}
        \sum_{ \ell = 1 }^{n-2}
        \tbinom{ n-1 }{ \ell }^{2} %
        \frac{ \Gamma( \frac{ \ell+1 }{ 2 } ) \Gamma( \frac{ n - \ell }{ 2 } ) }{ \Gamma(\frac{n}{2})^{2} }
        +
        \frac{ 2 n \pi^{\frac n 2} }{ (n+1) \Gamma(\tfrac{n}{2}) }
        .
    \end{align}
\end{theorem}
\begin{proof}
    The upper bound has already been shown. We address the lower bound.
    Suppose that $S$ has diameter equal to $1$. 
    Without loss of generality, after rotation, translation, and relabeling of vertices, 
    \begin{align*}
        v_{0} = 0, \quad v_{n} = e_{n}, \quad \| v_{0} - v_{n} \| = 1.
    \end{align*}
    Now, the remaining vertices $v_{1}, \dots, v_{n-1}$ can be written as 
    \begin{align*}
        v_{i} = ( f_{i}, \theta_{i} ), \qquad f_{i} \in \bbR^{n-1}, \qquad \theta_{i} \in \bbR.
    \end{align*}
    Since the edge from $v_{0}$ to $v_{n}$ realizes the diameter of $S$, for all $1 \leq i \leq n-1$:
    \begin{align*}
        f_{i} \neq 0, \qquad 0 < \theta_{i} < 1, 
        \qquad 
        \theta_{i}^{2} + \|f_{i}\|^{2} \leq 1, \qquad (1-\theta_{i})^{2} + \|f_{i}\|^{2} \leq 1
        .
    \end{align*}
    
    We write $\solidangle_{0} := \solidangle_{0}(S)$ and $\solidangle_{n} := \solidangle_{n}(S)$ as integrals.
    The vertices $v_{1},\dots,v_{n-1}$ are located in the northern half-space $\halfspace_{n}$. 
    Let $\unitv_{1}, \dots, \unitv_{n-1}$ be their radial projections on the northern hemisphere $\halfspace_{n} \cap \unitsphere{n-1}$, so that $\unitv_i = v_i / \|v_i\|$. 
    Their gnomonic projections onto the northern tangential plane $\bbE_{n-1}$ are 
    \begin{align*}
        \Gnomonic(\unitv_{i}) = ( f_{i} / \theta_{i}, 1 ) \in \bbE_{n-1}.
    \end{align*}
    We take the gnomonic projection of the spherical simplex at $v_{0}$ to the northern tangential plane $\bbE_{n-1}$. 
    Omitting the last coordinate, the image of the spherical simplex is the lower-dimensional simplex
    \begin{gather}
        K_{0}
        :=
        \convex\left\{
            0,
            \frac{f_{1}}{\theta_{1}},
            \dots,
            \frac{f_{n-1}}{\theta_{n-1}}
        \right\}
        \subset 
        \bbR^{n-1}
        .
    \end{gather}
    Therefore, 
    \begin{gather}
        \solidangle_{0} = \int_{K_{0}} ( 1 + \|z\|^2)^{-\frac{n}{2}} dz.
    \end{gather}
    A completely analogous argument applies to the solid angle $\solidangle_{n}$ at the vertex $v_{n}$, 
    with the main change being that factors $\theta_{i}$ are replaced by $1-\theta_{i}$. 
    Accordingly, 
    \begin{align*}
        \solidangle_n
        =
        \int_{K_{n}}( 1 + \|z\|^2 )^{-\frac{n}{2}} dz,
        \qquad
        K_{n}
        :=
        \convex\left\{
            0, \frac{f_1}{1-\theta_1}, \dots, \frac{f_{n-1}}{1-\theta_{n-1}}
        \right\}.
    \end{align*}

    Next, we consider any partition of the indices $\{1,\dots,n-1\}$ into two sets:
    let $I_A, I_B \subseteq \{1,\dots,n-1\}$ with 
    \begin{gather} \label{math:index-partition}
        I_{A} \cap I_{B} = \emptyset, 
        \qquad 
        I_{A} \cup I_{B} = \{ 1, \dots, n-1 \},
    \end{gather} 
    and define cardinalities 
    \begin{gather*} 
        \ell_{A} := |I_{A}|, 
        \qquad 
        \ell_{B} := |I_{B}|.
    \end{gather*}
    Notice that $\ell_{A} + \ell_{B} = {n-1}$. 
    We define the subspaces 
    \begin{align*}
        U_{A} := \linhull\{ f_i \suchthat i \in I_{A} \}, 
        \qquad
        U_{B} := \linhull\{ f_i \suchthat i \in I_{B} \},
    \end{align*}
    and the orthogonal projections onto their respective orthogonal complements 
    \begin{align*}
        P_{A} : \bbR^{n-1} \to U_{A}^\perp,
        \qquad 
        P_{B} : \bbR^{n-1} \to U_{B}^\perp.
    \end{align*}
    Since $\volume{S} > 0$, the vectors $f_1, \dots, f_{n-1}$ are linearly independent.
    Hence, $\dim U_{A} = \ell_{A}$ and $\dim U_{B} = \ell_{B}$. 
    
    We also define the following simplices and their corresponding volumes: 
    \begin{gather*}
        F     := \convex\{ 0, f_1, \dots, f_{n-1} \},
        \qquad 
        F_{A} := \convex\{ 0, f_{i} \suchthat i \in I_{A} \},
        \qquad 
        F_{B} := \convex\{ 0, f_{i} \suchthat i \in I_{B} \},
        \\
        \volume{F} := \Hausdorff^{n-1}     ( F     ),
        \qquad 
        \volumesymbol_{A} := \Hausdorff^{\ell_{A}}( F_{A} ),
        \qquad 
        \volumesymbol_{B} := \Hausdorff^{\ell_{B}}( F_{B} ).
    \end{gather*}

    Due to Cavalieri's principle, $S$ has the same measure as the convex hull of $e_{n}$ and $F$ (embedded in the coordinate plane $e_{n}^{\perp}$).
    Hence, $n! \volume{S} = (n-1)! \volume{F}$.
    Again via Cavalieri's principle, 
    $F$ has the same measure as the convex hull of $F_{A}$ and the projected simplex $P_{A} F_{B}$.
    Similarly, 
    $F$ has the same measure as the convex hull of $F_{B}$ and the projected simplex $P_{B} F_{A}$.
    Consequently, 
    \begin{align*}
        \Hausdorff^{\ell_{A}}( F_{A} ) \Hausdorff^{\ell_{B}}( P_{A} F_{B} ) = \frac{(n-1)!}{\ell_{A}! \ell_{B}!} \volume{F},
        \qquad 
        \Hausdorff^{\ell_{A}}( P_{B} F_{A} ) \Hausdorff^{\ell_{B}}( F_{B} ) = \frac{(n-1)!}{\ell_{A}! \ell_{B}!} \volume{F}.
    \end{align*}
    In particular, 
    \begin{align}\label{math:measure-of-projected-subsimplices}
        \Hausdorff^{\ell_{B}}( P_{A} F_{B} ) = \frac{(n-1)!}{\ell_{A}! \ell_{B}!} \frac{ \volume{F} }{ \volumesymbol_{A} },
        \qquad 
        \Hausdorff^{\ell_{A}}( P_{B} F_{A} ) = \frac{(n-1)!}{\ell_{A}! \ell_{B}!} \frac{ \volume{F} }{ \volumesymbol_{B} }.
    \end{align}
    Since the projections are orthogonal, 
    \begin{align}\label{math:upper-bounds-of-projected-simplices}
        \Hausdorff^{\ell_{B}}( P_{A} F_{B} ) \leq \volumesymbol_{B},
        \qquad 
        \Hausdorff^{\ell_{A}}( P_{B} F_{A} ) \leq \volumesymbol_{A}, 
        \qquad 
        \volume{F} \leq \frac{\ell_{A}! \ell_{B}!}{(n-1)!} \volumesymbol_{A} \volumesymbol_{B}.
    \end{align}
    
    This gives rise to estimates for the projected simplices $P_{B} K_{0}$ and $P_{A} K_{n}$. 
    Recall that $K_{0}$ is the convex hull of the origin and $f_{1} / \theta_{1}, \dots, f_{n-1} / \theta_{n-1}$. 
    Therefore, $P_{B} K_{0}$ is the convex hull of the origin and $P_{B} f_{i} / \theta_{i}$, where $i \in I_{A}$. 
    Thus, 
    \begin{align}\label{math:estimate-projected-K:A}
        \Hausdorff^{\ell_{A}}( P_{B} K_{0} ) 
        =
        \prod_{i \in I_A} \theta_{i}^{-1}
        \Hausdorff^{\ell_{A}}( P_{B} F_A ) 
        .
    \end{align}
    Along a completely analogous line of reasoning, one obtains 
    \begin{align}\label{math:estimate-projected-K:B}
        \Hausdorff^{\ell_{B}}( P_{A} K_n ) 
        = 
        \prod_{i \in I_B} (1-\theta_{i})^{-1}
        \Hausdorff^{\ell_{B}}( P_{A} F_B ) 
        .
    \end{align}
    
    We finally bound the integrals. 
    Obviously, $K_{0} \subseteq P_{B} K_{0} + U_{B}$. 
    The latter is an orthogonal sum. 
    Fubini's theorem gives
    \begin{align}\label{math:first-integral-estimate}
        \solidangle_{0}
        =
        \int_{K_{0}} ( 1 + \|z\|^2)^{-\frac{n}{2}} dz
        &
        \leq
        \int_{ P_{B} K_{0} } 
        \int_{ U_{B} }
        ( 1 + \|x\|^2 + \|y\|^2 )^{-\frac{n}{2}} 
        dy dx
        .
    \end{align}
    Recall the following Laplace transform identity: %
    \begin{align}\label{math:laplace-transform-identity}
        s^{-\frac{n}{2}} = \Gamma\!\left(\tfrac{n}{2}\right)^{-1} \int_{0}^{\infty} \tau^{\frac{n}{2} - 1} e^{-s\tau} d\tau
        .
    \end{align}
    Consider any $x \in P_{B} K_{0}$. We apply~\eqref{math:laplace-transform-identity} and find 
    \begin{align*}
        \int_{U_{B}} 
        ( 1 + \|x\|^2 + \|y\|^2 )^{-\frac{n}{2}} 
        dy
        &
        =
        \Gamma\!\left(\tfrac{n}{2}\right)^{-1} 
        \int_{0}^{\infty} \int_{\bbR^{\ell_{B}}}
        \tau^{\frac{n}{2} - 1} e^{-(1 + \|x\|^2 + \|y\|^2)\tau} 
        dy d\tau
        \\&
        =
        \Gamma\!\left(\tfrac{n}{2}\right)^{-1} 
        \int_{0}^{\infty} 
        \tau^{\frac{n}{2} - 1} e^{-(1 + \|x\|^2)\tau} 
        \int_{\bbR^{\ell_{B}}}
        e^{- \tau \|y\|^2 } 
        dy
        d\tau
        \\&
        =
        \Gamma\!\left(\tfrac{n}{2}\right)^{-1} 
        \int_{0}^{\infty} 
        \tau^{\frac{n-\ell_{B}}{2} - 1} e^{-(1 + \|x\|^2)\tau} d\tau
        \int_{\bbR^{\ell_{B}}}
        e^{- \|y\|^2 } dy
        .
    \end{align*}
    Using the Laplace transform identity~\eqref{math:laplace-transform-identity} once again and integrating the Gaussian yields 
    \begin{align*}
        \int_{\bbR^{\ell_{B}}}( 1 + \|x\|^2 + \|y\|^2 )^{-\frac{n}{2}} dy
        =
        \pi^{\frac{ \ell_{B} }{ 2 }}
        \Gamma\!\left(\tfrac{n}{2}\right)^{-1} 
        \Gamma\!\left(\tfrac{n-\ell_{B}}{2}\right)
        ( 1 + \|x\|^{2} )^{ -\frac{ n-\ell_{B} }{2} }
        .
    \end{align*}
    With $n - \ell_{B} = \ell_{A} + 1$, this results in %
    \begin{align}\label{math:solid-angle-0}
        \solidangle_{0}
        &
        \leq
        \pi^{\frac{ \ell_{B} }{ 2 }}
        \Gamma\!\left(\tfrac{n}{2}\right)^{-1} 
        \Gamma\!\left(\tfrac{\ell_{A}+1}{2}\right)
        \Hausdorff^{\ell_{A}}( P_{B} K_{0} )
        .
    \end{align}
    For the other solid angle $\solidangle_{n}$, a fully analogous integration argument with the two index sets switched yields
    \begin{align}\label{math:solid-angle-n}
        \solidangle_n
        &
        \leq
        \pi^{\frac{ \ell_{A} }{ 2 }}
        \Gamma\!\left(\tfrac{n}{2}\right)^{-1} 
        \Gamma\!\left(\tfrac{\ell_{B}+1}{2}\right)
        \Hausdorff^{\ell_{B}}( P_{A} K_n )
        .
    \end{align} 
    We combine~\eqref{math:measure-of-projected-subsimplices},~\eqref{math:upper-bounds-of-projected-simplices},~\eqref{math:estimate-projected-K:A},~\eqref{math:estimate-projected-K:B},~\eqref{math:solid-angle-0}, and~\eqref{math:solid-angle-n} and recall $\ell_{A} + \ell_{B} = n - 1$. 
    Now, 
    \begin{align}\label{math:estimate-per-partition}
        \minsolidangle{S}^{2} 
        \leq 
        \solidangle_{0} \solidangle_n 
        \leq 
        \frac{ 
            \pi^{\frac{ n-1 }{ 2 }}
            \Gamma( \frac{ \ell_{A}+1 }{ 2 } ) \Gamma( \frac{ \ell_{B}+1 }{ 2 } )
            \Gamma(\frac{n}{2})^{-2}
        }{
            \prod_{i \in I_A} \theta_{i}
            \prod_{i \in I_B} (1-\theta_{i})
        }
        \left( \frac{(n-1)!}{\ell_{A}! \ell_{B}!} \right)^{2}
        \frac{\volumesymbol_{F}^2}{\volumesymbol_{A} \volumesymbol_{B}},
        \qquad 
        \frac{\volumesymbol_{F}^2}{\volumesymbol_{A} \volumesymbol_{B}}
        \leq 
        \frac{\ell_{A}! \ell_{B}!}{(n-1)!}
        n 
        \volume{S}
        .
    \end{align}
    An improvement involves Theorem~\ref{theorem:sum-of-solid-angles} and concerns the extremal cases of the index partitions. 
    If $\ell_{A} = 0$, then $P_{A} K_n = K_n$ and 
    \begin{align*}
        \minsolidangle{S}^{2} 
        \leq 
        \minsolidangle{S}
        \solidangle{n}
        \leq 
        \frac{\unitspherearea{n-1}}{2n+2}
        \Hausdorff^{n-1}( K_n )
        =
        \frac{\unitspherearea{n-1}}{2n+2}
        \prod_{1 \leq i \leq n-1} (1-\theta_{i})^{-1}
        \volume{F}
        = 
        \frac{ \pi^{\frac n 2} \Gamma(\tfrac{n}{2})^{-1} (n+1)^{-1} n }{ \prod_{1 \leq i \leq n-1} (1-\theta_{i}) }
        \volume{S}
        .
    \end{align*}
    If $\ell_{A} = n-1$, then $P_{B} K_0 = K_0$ and 
    \begin{align*}
        \minsolidangle{S}^{2} 
        \leq 
        \frac{ \pi^{\frac n 2} \Gamma(\tfrac{n}{2})^{-1} (n+1)^{-1} n }{ \prod_{1 \leq i \leq n-1} \theta_{i} }
        \volume{S}
        .
    \end{align*}

    We sum over all the possible partitions~\eqref{math:index-partition}. 
    Since the coefficients only depend on the cardinalities of the index sets, 
    we arrive at the upper bound
    \begin{align*}
        \minsolidangle{S}^{2}
        \leq 
        C_{n}
        \volume{S}
        ,
        \qquad 
        C_{n}
        := 
        n 
        \pi^{\frac{ n-1 }{ 2 }}
        \sum_{ \ell = 1 }^{n-2}
        \tbinom{ n-1 }{ \ell }^{2} %
        \frac{ \Gamma( \frac{ \ell+1 }{ 2 } ) \Gamma( \frac{ n - \ell }{ 2 } ) }{ \Gamma(\frac{n}{2})^{2} }
        +
        \frac{ 2 n \pi^{\frac n 2} }{ (n+1) \Gamma(\tfrac{n}{2}) }
        .
    \end{align*}
    Since $\diameter{S} = 1$, we have $\volume{S} = \fullness{S}$ 
    and thus~\eqref{math:main-result} follows. 
    For general simplices with arbitrary diameters, the desired result follows by scaling invariance of $\fullness{S}$ and of the solid angles. 
\end{proof}

\begin{remark}
    A simplified, albeit rougher, estimate is of interest. 
    With the notation as in the proof of Theorem~\ref{theorem:main-result}, we consider the particular partition 
    \begin{align*}
        I_{A}:= \{ 1 \leq i \leq n - 1 \suchthat \theta_{i} > \tfrac{1}{2} \}, \qquad I_{B} := \{ 1 \leq i \leq n - 1 \suchthat \theta_{i} \leq \tfrac{1}{2} \}.
    \end{align*}
    Using this partition in~\eqref{math:estimate-per-partition}, this delivers the upper bound 
    \begin{align}
        \minsolidangle{S}^{2} 
        \leq 
        \solidangle_{0} \solidangle_n 
        \leq 
        2^{n-1} 
        \pi^{\frac{ n-1 }{ 2 }}
        \Gamma( \tfrac{ \ell_{A}+1 }{ 2 } ) \Gamma( \tfrac{ \ell_{B}+1 }{ 2 } )
        \Gamma(\tfrac{n}{2})^{-2}
        \frac{n!}{\ell_{A}! \ell_{B}!}
        \volume{S}
        .
    \end{align}
    By taking the maximum over $0 \leq \ell_{A} \leq n-1$, the resulting constant only depends on $n$. 
    The preceding proof can be interpreted as a refinement of this initial rough estimate. 
\end{remark}

\begin{example}\label{example:optimality-of-quadratic-exponent-in-all-dimensions}
    The following example illustrates the optimality of the quadratic exponent in Theorem~\ref{theorem:main-result} for simplices of any dimension $n \geq 3$. 
    Let $\epsilon > 0$ be sufficiently small. 
    The vertices 
    \begin{gather*}
        v_{0} = 0, \quad v_{1} = \epsilon e_{1}, 
        \quad 
        \dots 
        \quad v_{n-2} = \epsilon e_{n-2}, 
        \quad 
        v_{n-1} = e_{n-1} - \epsilon^{n-2} e_{n},
        \quad 
        v_{n  } = e_{n-1} + \epsilon^{n-2} e_{n}
    \end{gather*}
    define an $n$-dimensional simplex.
    Simple direct computations show that 
    \begin{align*}
        \diameter{S} = \sqrt{ 1 + \epsilon^{2} + \epsilon^{2(n-2)}} %
        ,
        \qquad 
        \volume{S} = \frac{2}{n!} \epsilon^{2(n-2)}
        ,
        \qquad 
        \fullness{S} = \frac{ 2 \epsilon^{2(n-2)} }{ n! ( 1 + \epsilon^{2} + \epsilon^{2(n-2)} )^{\frac n 2} }
        .
    \end{align*}
    By symmetry, 
    \begin{align*}
        \solidangle_{0}(S)
        =
        \frac{ \unitspherearea{n-1} }{ 2^{n-2} \pi } 
        \arctan( \epsilon^{n-2} )
        .
    \end{align*}
    For $\epsilon > 0$ small enough, 
    \begin{align*} 
        \frac{ \epsilon^{n-2} }{ 2^{n} } 
        \unitspherearea{n-1}
        \leq 
        \solidangle_{0} 
        \leq 
        \frac{ \unitspherearea{n-1} }{ 2^{n-2} \pi } \epsilon^{n-2}
        .
    \end{align*}

    By symmetry, too, $\solidangle_{1}(S) = \dots = \solidangle_{n-2}(S)$ and $\solidangle_{n-1}(S) = \solidangle_{n}(S)$.
    Let 
    \begin{align*}
        M_{1}
        :=
        \begin{pmatrix}
            -\epsilon & -\epsilon & \dots    & -\epsilon & -\epsilon       & -\epsilon 
            \\
                      & \epsilon  &          &           &                 &
            \\
                      &           & \ddots   &           &                 &
            \\
                      &           &          &  \epsilon &                 &
            \\
                      &           &          &           &  1              & 1
            \\
                      &           &          &           & -\epsilon^{n-2} & \epsilon^{n-2}            
        \end{pmatrix}
        ,
        \qquad 
        M_{n}
        :=
        \begin{pmatrix}
                             & \epsilon        &                 &                 &             
            \\
                             &                 & \ddots          &                 &           
            \\
                             &                 &                 &  \epsilon       &           
            \\
             -1              &  -1             & \dots           &  -1             &                    
            \\
             -\epsilon^{n-2} & -\epsilon^{n-2} & \dots           & -\epsilon^{n-2} & -2\epsilon^{n-2} 
        \end{pmatrix}
    \end{align*}
    be the matrices whose columns are the edges of the simplex $S$ that emanate from the vertices $v_{1}$ and $v_{n}$, respectively. Their determinants have absolute value $2 \epsilon^{2(n-2)}$. 
    For any $x \in \bbR^{n}$ with positive entries, we notice 
    \begin{align*}
        \| M_{1} x \|^{2} 
        &
        = 
        \epsilon^{2} ( \tsum_{i=1}^{n} x_{i} )^{2} 
        + 
        \epsilon^{2} \tsum_{i=2}^{n-2} x_{i}^{2} 
        + 
        ( x_{n-1} + x_{n} )^{2} 
        + 
        \epsilon^{2(n-2)} ( x_{n-1} - x_{n} )^{2}
        .
    \end{align*}
    Assuming that $\epsilon > 0$ is small, the following simple estimates hold:
    \begin{align*}
        \| M_{1} x \|^{2} 
        &
        \leq  
        n \epsilon^{2} \tsum_{i=1}^{n} x_{i}^{2} 
        + 
        \epsilon^{2} \tsum_{i=1}^{n-2} x_{i}^{2} 
        + 
        ( 1 + \epsilon^{2(n-2)} ) ( x_{n-1} + x_{n} )^{2} 
        \\&
        \leq  
        ( 1 + n ) \epsilon^{2} \tsum_{i=1}^{n-2} x_{i}^{2} 
        + 
        n \epsilon^{2} ( x_{n-1}^{2} + x_{n}^{2} )
        + 
        ( 1 + \epsilon^{2(n-2)} ) ( x_{n-1} + x_{n} )^{2} 
        \\&
        \leq  
        ( 1 + n ) \epsilon^{2} \tsum_{i=1}^{n-2} x_{i}^{2} 
        + 
        ( 1 + \epsilon^{2(n-2)} + n \epsilon^{2} ) ( x_{n-1} + x_{n} )^{2} 
        ,
        \\
        \| M_{1} x \|^{2} 
        &
        \geq  
        \epsilon^{2} ( \tsum_{i=1}^{n-2} x_{i}^{2} ) 
        + 
        ( x_{n-1}^{2} + x_{n}^{2} )
        .
    \end{align*}
    Applying the lower and upper bound in the representation~\eqref{math:orthant-integral} of the solid angles as orthant integral and performing a substitution,
    together with $n! \volume{S} = 2 \epsilon^{2(n-2)}$, yields
    \begin{align*}
        \frac{
            \unitspherearea{n-1}\epsilon^{n-2}
        }{
            2^n
            (n+1)^{\frac{n-2}{2}}
            (1+n\epsilon^2+\epsilon^{2(n-2)})
        }
        \leq
        \solidangle_1(S)
        \leq
        \frac{
            \unitspherearea{n-1}
        }{
            2^{n-1}
        }
        \epsilon^{n-2}.
    \end{align*}
    Similarly, the norm of $M_{n} x$ expands like 
    \begin{align*}
        \| M_{n} x \|^{2} 
        &
        = 
        \epsilon^{2} \tsum_{i=2}^{n-1} x_{i}^{2} 
        + 
        ( \tsum_{i=1}^{n-1} x_{i} )^{2} 
        + 
        \epsilon^{2(n-2)} ( \tsum_{i=1}^{n-1} x_{i} + 2 x_{n} )^{2}
        ,
    \end{align*}
    and we derive the estimates 
    \begin{align*}
        \| M_{n} x \|^2
        &
        \leq
        \epsilon^2
        \tsum_{i=2}^{n-1} x_{i}^2
        +
        ( 1 + 2 \epsilon^{2(n-2)} )
        ( \tsum_{i=1}^{n-1} x_{i} )^2
        +
        8 \epsilon^{2(n-2)} x_{n}^2
        \\&
        \leq
        ( n - 1 + \epsilon^2 + 2(n-1)\epsilon^{2(n-2)} )
        \tsum_{i=1}^{n-1} x_{i}^2
        +
        8 \epsilon^{2(n-2)} x_{n}^2,
        \\
        \| M_{n} x \|^{2} 
        &
        \geq 
        \tsum_{i=1}^{n-1} x_{i}^{2} + \epsilon^{2(n-2)} 4 x_{n}^{2}
        .
    \end{align*}
    Again applying the lower and upper bound in the representation~\eqref{math:orthant-integral} results in the following lower and upper bound for the solid angle:
    \begin{align*}
        \frac{
            \unitspherearea{n-1}
            \epsilon^{n-2}
        }{
            2^n \sqrt{2}
            ( n - 1 + \epsilon^2 + 2(n-1) \epsilon^{2(n-2)} )^{\frac{n-1}{2}}
        }
        \leq
        \solidangle_n(S)
        \leq
        \frac{
            \unitspherearea{n-1}
        }{
            2^n
        }
        \epsilon^{n-2}.
    \end{align*}
    We conclude that there exist constants $c_{n}, C_{n} > 0$, depending only on $n$, such that 
    \begin{align*}
        c_{n} \minsolidangle{S}^{2} \leq \volume{S} \leq C_{n} \minsolidangle{S}^{2}
        .
    \end{align*}
    This illustrates that the quadratic exponent $2$ is optimal.
\end{example}

We finish this discussion with low-dimensional examples. 
Triangles admit an improved bound where the fullness is bounded not only from above but also from below by multiples of the minimum (solid) angle.
Tetrahedral minimum angle conditions were studied by Liu and Joe~\cite{liu1994relationship},
and both Theorem~\ref{theorem:tetrahedral-main-result} and Theorem~\ref{theorem:main-result} for $n=3$ yield improved bounds. 
Lastly, the minimum solid angle condition for pentatopes is illustrated. 

\begin{example}
    Let $S$ be a triangle with angles $\alpha, \beta, \gamma \in (0,\pi)$ and opposite side lengths $a, b, c > 0$, respectively. 
    Assume that $\alpha$ equals the minimal angle and that $c$ equals the diameter. 
    Using the law of sines, 
    \begin{align*}
        2 \volume{S} 
        = 
        a b \sin(\gamma) 
        =
        c^{2} \frac{ \sin(\alpha) \sin(\beta) }{ \sin(\gamma) }
        = 
        c^{2} \sin(\alpha) \frac{ \sin(\beta) }{ \sin(\alpha + \beta) }
        .
    \end{align*}
    Since $\alpha \leq \beta \leq \pi-\alpha$ and $3\alpha \leq \pi$, 
    we have $\sin(\alpha) \leq \sin(\beta)$ and $\cos(\beta) \leq \cos(\alpha)$.
    From $b \leq c$ we get $\sin(\beta) \leq \sin(\gamma)$, and since $\beta + \gamma < \pi$ we get $\beta \leq \gamma$, hence $\cos(\beta) \geq 0$.
    By the angle sum identity,
    \begin{align*}
        \frac{\sin(\alpha+\beta)}{\sin(\beta)}
        = 
        \frac{\sin(\alpha)\cos(\beta) + \cos(\alpha)\sin(\beta)}{\sin(\beta)}
        \leq 
        \cos(\beta) + \cos(\alpha)
        \leq 
        2 \cos(\alpha)
        .
    \end{align*}
    With $\alpha \leq \tan(\alpha)$ and $\sin(\alpha) \leq \alpha$, we find 
    \begin{align*}
         \tfrac{1}{4} \alpha \leq \tfrac{1}{4} \tan(\alpha)
         \leq 
         \frac{\volume{S}}{\diameter{S}^{2}} \leq \tfrac{1}{2} \sin(\alpha) \leq \tfrac{1}{2} \alpha.
    \end{align*}
    This proves the minimum angle of a triangle to be a shape measure comparable to the fullness.
    Using Theorem~\ref{theorem:main-result} and $\alpha \leq \pi/3$, we get the weaker estimates 
    \begin{align}
        \tfrac{ 3 }{ 4 \pi } 
        \alpha^{2}
        \leq 
        \fullness{S}
        \leq 
        \tfrac{1}{2} \alpha 
        ,
        \qquad 
        \tfrac{ 3 }{ 4 \pi } 
        \alpha^{2}
        \leq 
        \tfrac{1}{4}
        \alpha
        .
    \end{align}
    Thus, Theorem~\ref{theorem:main-result} recovers natural inequalities relating the triangle fullness and minimum angle.
\end{example}

\begin{example}
    We consider a tetrahedron $S$ with vertices $v_{0}, v_1, v_2, v_{3}$. 
    Liu and Joe~\cite{liu1994relationship} compare shape measures of tetrahedra,
    including the minimum solid angle $\minsolidangle{S}$ and the \emph{mean ratio} $\meanratio{S}$,
    where~\cite[(3.4)]{liu1994relationship}  
    \begin{align*}
        \meanratio{S} := 
        12 ( 3 \volume{S} )^{\frac{2}{3}} \left( \tsum_{0 \leq i < j \leq 3} \| v_i - v_j \|^{2} \right)^{-1}
        .
    \end{align*}
    This implies a comparison between the fullness and the minimum solid angle of tetrahedra, as we now show.
    Without loss of generality, $\diameter{S} = \| v_{0} - v_1 \|$. Whereas $\| v_2 - v_{3} \|$ can be arbitrarily small, 
    \begin{align*}
        \| v_{0} - v_2 \|^{2} + \| v_2 - v_1 \|^{2} + \| v_{0} - v_{3} \|^{2} + \| v_{3} - v_1 \|^{2} 
        \geq 
        \diameter{S}^{2}
        .
    \end{align*}
    Therefore, the fullness and the mean ratio satisfy the comparison 
    \begin{align*}
        2^{\frac{3}{2}} 3 \fullness{S}
        \leq 
        \meanratio{S}^{\frac{3}{2}} 
        \leq 
        6^{\frac{3}{2}} 3 \fullness{S}
        .
    \end{align*}
    Define $\sineofminsolidangle{S} := \sin( \minsolidangle{S} / 2 )$. 
    By Liu and Joe's comparison result~\cite[Theorem~5.2]{liu1994relationship}, %
    \begin{align*}
        \tfrac{1}{16} \meanratio{S}^{\frac{3}{2}} \leq \sineofminsolidangle{S} \leq 8^{\frac 1 4} \meanratio{S}^{\frac 3 4}
        .
    \end{align*}
    The sum of the four solid angles in $S$ is at most $2\pi$, 
    thus $\minsolidangle{S} \leq \pi / 2$.
    The inequalities
        $
        \tfrac{ \sqrt{2} }{ \pi } \minsolidangle{S} 
        \leq 
        \sineofminsolidangle{S} 
        \leq 
        \tfrac{1}{2} \minsolidangle{S}
        $
    follow. 
    The combination of these findings gives 
    \begin{align}\label{math:joe-liu-fullness-vs-angle}
        3 \frac{ \sqrt{2} }{4} \fullness{S}
        \leq 
        \minsolidangle{S}
        ,
        \qquad 
        \minsolidangle{S}^{2}
        \leq 
        36
        \sqrt{3}
        \pi^{2} 
        \fullness{S}
        .
    \end{align}
    We use Theorem~\ref{theorem:main-result} with $n=3$ and the fact $\Gamma(\tfrac{3}{2}) = \sqrt{\pi}/2$ to find the stronger estimates 
    \begin{align}
        3 \fullness{S}
        \leq 
        \minsolidangle{S}
        ,
        \qquad 
        \minsolidangle{S}^{2}
        \leq 
        ( 48 + 3 \pi )
        \fullness{S}
        .
    \end{align}
    In addition, Theorem~\ref{theorem:tetrahedral-main-result} further improves the constant and yields
    \begin{align}
        3 \fullness{S}
        \leq 
        \minsolidangle{S}
        ,
        \qquad 
        \minsolidangle{S}^{2}
        \leq 
        ( 12 + 3 \pi )
        \fullness{S}
        .
    \end{align}
    In particular, while a comparison between fullness and minimum solid angle of a tetrahedron already follows from known results,
    this manuscript yields tighter constants for this comparison of shape measures. 
\end{example}

\begin{example}
    If $S \subset \bbR^{4}$ is a $4$-simplex (pentatope), then Theorem~\ref{theorem:main-result} yields 
    \begin{align}
        4 \fullness{S}
        \leq 
        \minsolidangle{S}
        ,
        \qquad 
        \minsolidangle{S}^{2} 
        \leq 
        \frac{188}{5} \pi^{2}
        \fullness{S}
        .
    \end{align}
\end{example}

\end{document}